\documentclass[11pt]{amsart}

\usepackage{amssymb}
\usepackage{graphicx}
\usepackage{tikz-cd}
\usepackage{bm}

\usepackage[all,knot]{xy}
\usepackage{tikz-cd}
\usepackage{enumitem}
\usepackage{multirow}
\usepackage{cancel}

\newcommand{\mfo}{\mathfrak{o}}
\newcommand{\mft}{\mathfrak{t}}

\newcommand{\kk}{\mathrm{k}}
\newcommand{\hh}{\mathbf{h}}

\newcommand{\ot}{\otimes} 
\newcommand{\ott}{\bar{\otimes}}

\DeclareMathOperator{\HH}{HH}
\DeclareMathOperator{\HHom}{Hom}

\newtheorem{theorem}{Theorem}[section]
\newtheorem{lemma}[theorem]{Lemma}
\newtheorem{Proposition}[theorem]{Proposition}

\theoremstyle{definition}
\newtheorem{definition}[theorem]{Definition}
\newtheorem{example}[theorem]{Example}

\theoremstyle{remark}
\newtheorem{remark}[theorem]{Remark}

\numberwithin{equation}{section}

\title[Cluster-tilted algebra]{Comultiplicative map on projective resolution for a family of algebras one of which is cluster-tilted type $\mathbb{D}_4$}

\author[P.\ Mishra]{Pratyush Mishra}
\address{Wake Forest University, 127 Manchester Hall, Winston-Salem, NC 27109, USA}
\email{mishrap@wfu.edu}

\author[T.\ Oke]{Tolulope Oke}
\address{Wake Forest University, 127 Manchester Hall, Winston-Salem, NC 27109, USA}
\email{oket@wfu.edu}

\date{\today}

\usepackage{amsmath}
\usepackage{hyperref}
\usepackage{ccaption}

\usepackage{cleveref}


\thanks{}

\date{\today}

\begin{document}


\begin{abstract}
We construct a comultiplicative map on the projective bimodule resolution for a family of algebras one of which is cluster-tilted of type $D_4$. The comultiplicative map presented in terms of idempotents associated with vertices of the quiver and a chain homotopy map is used to describe the multiplicative structure on the Hochschild cohomology ring of the first member of the family. We define a star product used to describe the cup product structure on the Hochschild cohomology ring for all family members.

\end{abstract}

\maketitle

\let\thefootnote\relax

\section{Introduction}\label{introduction}

Let $A$ be a finite dimensional $\kk$-algebra. It is of interest in general to understand the representations of the algebra $A$. The use of homological algebra tools such as finding projective resolutions of $A$, constructing diagonal maps on resolutions, finding associated homology or cohomology groups and extracting meaningful information have provided researchers with several approaches towards understanding algebra representations. This often leads to several questions about how such resolutions may be constructed for example. If the associated projective modules used in the resolution are both left and right modules over the algebra, the resolution is regarded as a bimodule resolution. Several articles have been written to study algebras in which the associated modules in the projective resolution repeat at different degrees. Such algebras are called \textit{periodic algebras}.

Examples of periodic algebras are preprojective algebras associated with Dynkin graphs \cite{DUGAS2010990}, trivial extension algebras of path algebras of Dynkin quivers \cite{Brenner2002}, and algebras associated with cyclic quivers and separable algebras \cite{10.55937/sut/1252509532}. A. B. Buan, R.J. Marsh and I. Reiten in \cite{5a735f82-d5f0-3634-8b0d-9ab687ac9597} introduced cluster-tilted algebras as the endomorphism algebras of cluster-tilting objects in a cluster category. I. Assem, T. Br{\"u}stle, and R. Schiffler proved that an algebra $\tilde{C}$ is cluster-tilted if there is a tilted algebra $C$ (of global dimension at most two) for which $\tilde{C}$ is the trivial extension of $C$ \cite{assem2008cluster}. This suggests that cluster-tilted algebras also have periodic projective bimodule resolutions. 

The minimal projective bimodule resolution of an algebra is used in determining the ring and Gerstenhaber algebra structure on its Hochschild cohomology. There are several techniques used to determine the ring structure on Hochschild cohomology. For instance, the Yoneda product and the Yoneda splicing \cite{witherspoon2019hochschild}. The ring multiplication is usually referred to as the cup product and it is common to determine this product using a diagonal map on the projective bimodule resolution that one has chosen. We refer to the diagonal map as the comultiplicative map throughout this article. Basic homological algebra results ensure the existence of such a map under certain conditions. In practice, we often make a concrete construction of the comultiplication map for each algebra of interest. 

T. Furaya \cite{furuya2012projective} studied a class of algebras which are special biserial algebras one of which is a cluster-tilted algebra of type $D_4$. We construct a comultiplication map on the projective bimodule resolution for these class of algebras. Our construction is applicable to algebras of finite representation types having periodic projective bimodule resolutions or having \textit{almost periodic} projective resolutions of period three. By \textit{almost periodic}, we mean that the kernel of the first differential map has a periodic resolution. This can be a spring board for constructing explicit comultiplicative maps in terms of vertices of associated quivers for larger classes of algebras in future work.

We introduce quiver algebras and the family of algebras under study in Section \ref{preliminaries}. We also present other preliminary results on Hochschild cohomology ring of algebras, the projective bimodule resolutions $\mathbb{Q}_{\bullet}$ and useful module compatibility structure in \ref{preliminaries}. In Section \ref{Main results}, we present our main results describing the comultiplicative map for the family of algebras under study. In order to describe the multiplicative structure on Hochschild cohomology, we further present in Section \ref{Main results} a star product using certain chain map.

More specifically, and using the resolution $\mathbb{Q}_{\bullet}$, we construct a comultiplicative map $\Delta:\mathbb{Q}_{\bullet}\xrightarrow{}\mathbb{Q}_{\bullet}\ot_{\Lambda_n}\mathbb{Q}_{\bullet}$ where the tensor product is taken over the algebra $\Lambda_n$. Throughout, we distinguish between the tensor product over the field and this tensor product by defining $\ott:=\ot_{\Lambda_n}$. Our first result is presented in Lemma \ref{comultiplicative_map} where we defined the chain map $\Delta'$ lifting the map $\pi:\Lambda_n\rightarrow\Lambda_n$ given as multiplication by two. We then state our main result in Theorem \eqref{comultiplicative_map} where we define $\Delta = \Delta'+\hh\partial+d\hh$ for some chain homotopy map $\hh$. We use $\Delta$ to describe the cup product structure on Hochschild cohomology $\HH^*(\Lambda_0)$ of the cluster-tilted algebra of type $D_4$. Using the constructed chain map $\Delta'$, we present a table of multiplication $f\star g$ at the chain level for elements $f,g\in \HHom_{\Lambda_n}(\mathbb{Q}_{j},\Lambda_n)$ for $j=0, 3m, 3m+1, 3m+2$ in Theorem \eqref{comultiplicative_structure}. The ring structure on Hochschild cohomology groups $\HH^j(\Lambda_n)$ at the chain level can be obtain using Theorem \eqref{comultiplicative_structure}, the differential maps $\partial^{\bullet}$ and $d^{\bullet}$ and chain homotopy $\hh^{\bullet}$. We present some examples on how to define $\hh$ in Definition \ref{homotopy-equation}. As noted in \cite{KARADAG2021106597}, the comultiplicative map will be helpful in describing the Gerstenhaber algebra structure on Hochschild cohomology of these family of finite dimensional algebras.

\section{Preliminaries}\label{preliminaries}
\textbf{Quiver algebras:} A quiver is a directed graph. The quiver $\mathcal{Q}$ will sometimes be denoted as a quadruple $(\mathcal{Q}_0,\mathcal{Q}_1,\mathfrak{o},\mathfrak{t})$ where $\mathcal{Q}_0$ is the set of vertices in $\mathcal{Q}$, $\mathcal{Q}_1$ is the set of arrows in $\mathcal{Q}$, and $\mfo,\mft: \mathcal{Q}_1 \longrightarrow \mathcal{Q}_0$ are maps which assign to each arrow $a\in \mathcal{Q}_1$, its origin vertex $\mfo(a)$ and terminal vertex $\mft(a)$ in $\mathcal{Q}_0$. A sequence of paths is written $a = a_1 a_2 \cdots a_{n-1} a_n $ such that the terminal vertex of $a_{i}$ is the same as the origin vertex of $a_{i+1}$, using the convention of concatenating paths from left to right. Let $\mathrm{k}$ be a field. The quiver or path algebra $\mathrm{k}Q$ is defined as a vector space having all paths in $\mathcal{Q}$ as a basis. Vertices are regarded as paths of length $0$ or trivial paths, an arrow is a path of length $1$, and so on. The multiplication on $\mathrm{k}Q$ is concatenation of paths. Two paths $a$ and $b$ satisfy $ab=0$ if and only if $\mft(a)\neq \mfo(b)$. This multiplication defines an associative graded algebra over $\mathrm{k}$ given by $\mathrm{k}\mathcal{Q} = \bigoplus_{m\geq 0} \mathrm{k}\mathcal{Q}_m$. The path algebra is unital if and only if the quiver has only finitely many vertices. 

We now consider the following quiver $\mathcal{Q}$ which was studied in \cite{furuya2012projective}.

\begin{center}
$\mathcal{Q}$: \qquad \begin{tikzcd}
&   \arrow[dl,"a_1"]e_1 &  \\
e_2  \arrow[rr,"a_2= b_2"]&& e_0 \arrow[ul,"a_0"] \arrow[dl,"b_0"] \\
& \arrow[ul,"b_1"] f_1 & 
\end{tikzcd}
\end{center}

The vertices of $\mathcal{Q}$ are elements of the set of trivial paths $\{e_0,e_1,e_2,f_1\}$ and the arrows are given by the set $\{a_0, a_1, a_2=b_2, b_0, b_1\}$. Throughout, we will have all indices $i$ for $e_i,f_i, a_i$ and $b_i$ to be given modulo 3 where $f_0=e_0$ and $f_2=e_2$. The unital element for $\mathcal{Q}$ is $e_0+e_1+f_1+e_2$. 

We can take a linear combination of paths as long as they are parallel. Parallel paths have the same origin and terminal vertices. Such a linear combination is a uniform path or relation. We can take the ideal generated by a set of relations. In the quiver $\mathcal{Q}$, the relations $(a_0a_1a_2)^na_0a_1$ is a path of length $3n+2$ having origin vertex $\mfo(a_0)=e_0$ and terminal vertex $\mft(a_1)=e_2$. This path is parellel to $b_0b_1$ since they both have the same origin and terminal vertices. For each $n\geq 0$ and for $i=1,2$, we consider the ideal generated by the following relations
\begin{equation}\label{fam-ideals}
I_n =\langle (a_0a_1a_2)^na_0a_1-b_0b_1, (a_ia_{i+1}a_{i+2})^na_{i}a_{i+1}, b_ib_{i+1} \rangle
\end{equation}
and study the family of quiver algebras 
\begin{equation}\label{family}
\Lambda_n = \kk\mathcal{Q}/I_n.
\end{equation}

\textbf{What is known:}  Cluster algebras have direct connections to the theory of representation of quivers and they were introduced by Fomin and Zelevinsky in \cite{fomin2002cluster}. Cluster algebras are obtained from quivers and one way a cluster algebra is obtained from a quiver is using the so-called quiver mutations. A detail description of this technique is given in \cite[Section 3.2]{Keller_2010}. The notion of tilting is categorical and has to do with obtaining a cluster-tilted algebra from the category of cluster algebras obtained from a quiver. 

According to \cite{furuya2012projective}, the algebra $\Lambda_n$ is a special biserial algebra. One of the key features of special biserial algebras is that they have a very controlled representation theory. This means that their representations can be described in a relatively simple and structured way. An example showing that when $n=0$, the algebra $\Lambda_0$ is a cluster-tilted algebra of the Dynkin type $\mathbb{D}_4$ is given in \cite[Example 3.5]{assem2008cluster}. T. Furuya in \cite{furuya2012projective} adapted a construction of a projective bimodule resolutions for special biserial algebras described in \cite{snashall2010hochschild} to construct a projective bimodule resolution $\mathbb{Q}_{\bullet}$ for the family $\Lambda_n$. This is similar to the projective bimodule resolution constructed for Koszul algebras in \cite{green2005resolutions} and for a family of quiver algebras whose Hochschild cohomology modulo nilpotent elements is not finitely generated in \cite{Oke04052022}. The construction of the projective bimodule resolution required the construction of sets $\mathcal{G}^i$, $i\geq 0$ of uniform paths of length greater than or equal to $i$. The resolution was used to obtain the main result of \cite{furuya2012projective}.  That is, describe explicit generating sets for the Hochschild cohomology groups $\HH^j(\Lambda_n)$ for each $j$, and obtain a corresponding formula for the dimension of such vector space. 


\textbf{The Hochschild cohomology} of an associative $\kk$-algebra $A$ was originally defined using the following projective bimodule resolution known as the bar resolution. It is given by
\begin{equation}\label{bar}
\mathbb{B}_{\bullet}(A):  \qquad\cdots \rightarrow A^{\ot (m+2)}\xrightarrow{\;\delta_m\; }A^{\ot (m+1)}\xrightarrow{\delta_{m-1}} \cdots \xrightarrow{\;\delta_2\;}A^{\ot 3}\xrightarrow{\;\delta_1\;}A^{\ot 2}\;(\;\xrightarrow{\mu} A)
\end{equation}
where $\mu$ is multiplication and the differentials $\delta_m$ are given by
\begin{equation}\label{bar-diff}
\delta_n(a_0\otimes a_1\otimes\cdots\otimes a_{m+1}) = \sum_{i=0}^m (-1)^i a_0\otimes\cdots\otimes a_ia_{i+1}\otimes\cdots\otimes a_{m+1}
\end{equation}
for all $a_0, a_1,\ldots, a_{m+1}\in A$. Each $A$-bimodule $A^{\ot (m+1)}$ is a left modules over the enveloping algebra $A^e = A\ot A^{op}$, where $A^{op}\cong A$ with reverse multiplication as an algebra. The resolution is sometimes written $\mathbb{B}_{\bullet}\xrightarrow{\mu}\Lambda$ with $\mu$ referred to as the augmentation map. The Hochschild cohomology of $A$ with coefficients in $A$ denoted $ \HH^*(A)$ is obtained by applying the functor $\HHom_{A^e}(-,A)$ to the complex $\mathbb{B}_{\bullet}(A)$, and then taking the cohomology of the resulting cochain complex. That is 
$$ \HH^*(A) := \bigoplus_{j\geq 0} \HH^j(A)=\bigoplus_{j\geq 0} \text{H}^j(\HHom_{A^e}(\mathbb{B}_{m}(A),A)). $$

Let $\chi\in\HHom_{A^e}(\mathbb{B}_m,A)$ be a cocycle, that is, $(\delta^{*}(\chi))(\cdot) := \chi\delta(\cdot)=  0.$ There is an isomorphism of the abelian groups $\HHom_{A^e}(\mathbb{B}_m(A),A)\cong\HHom_{k}(A^{\otimes m},A),$ so we can also view $\chi$ as an element of $\HHom_{k}(A^{\otimes m},A)$.

\textbf{The cup product:} We obtain the multiplicative structure on Hochschild cohomology using the following.
\begin{definition}\label{cup-definition}
Let $A$ be a $\kk$-algebra. Let $\Delta_{\mathbb{B}}:\mathbb{B}\xrightarrow{}\mathbb{B}\ot_{A}\mathbb{B}$ be the comultiplicative map lifting the identity map on $A \cong A\ot_{A} A$. Let $f\in\HHom_{\kk}(A^{\otimes m},A)$ and $g\in \HHom_{\kk}(A^{\otimes n},A)$ be cocycles of degree $m$ and $n$ respectively. The cup product $f\smallsmile g$ at the chain level is an element of $\HHom_{\kk}(A^{\otimes (m+n)},A)$ given by
\begin{equation*}\label{cup-equ-1}
f\smallsmile g = \tilde{\pi} (f\ot g)\Delta_{\mathbb{B}},
\end{equation*}
where $\tilde{\pi}$ is multiplication, and $\Delta_{\mathbb{B}}$ is given by 
\begin{equation}\label{diag-bar}
\Delta_{\mathbb{B}}(a_0\otimes\cdots\otimes a_{n+1}) = \sum_{i=0}^n (a_0\otimes\cdots\otimes a_i\ot 1)\ot_{A}( 1\ot a_{i+1}\otimes\cdots\otimes a_{n+1}).
\end{equation}
\end{definition}

\textbf{The projective bimodule resolution $\mathbb{Q}_{\bullet}$:} We now describe the projective bimodule resolution $\mathbb{Q}_{\bullet}$ that was used to obtain the result of \cite[Theorem 4.10]{furuya2012projective}. 

For ease of notations and computations, in $\mathcal{Q}$, we label $e_0=f_0$, $e_2=f_2$ and $a_2=b_2$ so that a path of length three from $e_0$ to $e_0$ could be written as $a_0a_1a_2$ for the upper triangle of the quiver and as $b_0b_1b_2$ for the lower triangle of the quiver. Throughout, we will have all indices $i$ for $e_i,f_i, a_i$ and $b_i$ to be given modulo 3. 

We first define a set of uniform paths of length $i$ given as $\mathcal{G}^i$, for $i=0,1,2,\ldots$ in the following ways. Note that the path $(a_0a_1a_2)^na_0a_1$ of length $3n+2$ can also be written as $a_0(a_1a_2a_0)^na_1$.
\begin{gather*}
\mathcal{G}^0 =\{\text{vertices}\} =  \{R^0=e_0=f_0 , S^0=e_1, T^0=f_1, U^0=e_2=f_2\},\\
\mathcal{G}^1 =\{\text{arrows}\} =  \{R^1_0=a_0, R^1_1=b_0 , S^1=a_1, T^1=b_1, U^1=a_2=b_2\},\\
\mathcal{G}^2 =\{\text{minimal generating set for $I_n$}\} =  \{R^2=a_0(a_1a_2a_0)^na_1-b_0b_1, \\
S^2=a_1(a_2a_0a_1)^na_2, T^2=b_1b_2,  U^2_0=a_2(a_0a_1a_2)^na_0, U^2_1=b_2b_0\},
\end{gather*}
and in general,
\begin{equation}\label{uniform-paths}
\mathcal{G}^i = \begin{cases} \{R^i, S^i_0, S^i_1, T^i_0, T^i_1, U^i\}  & \text{ if } i\equiv 0 \pmod{3}, \\ 
\{R^i_0, R^i_1, S^i, T^i, U^i\}  & \text{ if } i\equiv 1 \pmod{3},\\
\{R^i, S^i, T^i, U^i_0, U^i_1\}  & \text{ if } i\equiv 2 \pmod{3}, \end{cases}
\end{equation}
where the uniform paths $\{R^i_j, S^i_j, T^i_j, U^i_j\}$ are defined recursively as follows;
\begin{enumerate}
\item[(R)] \begin{align*}
R^{6i+1}_j = \begin{cases} R^{6i}a_0 & \text{ if } j=0 \\  R^{6i}b_0 & \text{ if } j=1 \end{cases}, \quad  & R^{6i+2} = R^{6i+1}_0(a_1a_2a_0)^na_1 - R^{6i+1}_1b_1,\\
R^{6i+3} = R^{6i+2}a_2 = R^{6i+2}b_2, \quad & R^{6i+4}_j = \begin{cases} R^{6i+3}(a_0a_1a_2)^na_0 & \text{ if } j=0 \\  R^{6i+3}b_0 & \text{ if } j=1 \end{cases} \\
 R^{6i+5} = R^{6i+4}_0a_1 - R^{6i+4}_1b_1, \quad & R^{6i+6} = R^{6i+5}(a_2a_0a_1)^na_2.
\end{align*}
\item[(S)] \begin{align*}
 S^{6i+1} = S^{6i}_0a_1 - S^{6i}_1b_1, \quad & S^{6i+2} = S^{6i+1}(a_2a_0a_1)^na_2, \\
S^{6i+3}_j = \begin{cases} S^{6i+2}a_0 & \text{ if } j=0 \\  S^{6i+2}b_0 & \text{ if } j=1 \end{cases}, \quad  & S^{6i+4} = S^{6i+3}_0(a_1a_2a_0)^na_1 - S^{6i+3}_1b_1,\\
S^{6i+5} = S^{6i+4}a_2 = S^{6i+4}b_2, \quad & S^{6i+6}_j = \begin{cases} S^{6i+5}(a_0a_1a_2)^na_0 & \text{ if } j=0 \\  S^{6i+5}b_0 & \text{ if } j=1 \end{cases}.
\end{align*}
\end{enumerate}
 
\begin{enumerate}
\item[(T)] \begin{align*}
 T^{6i+1} = T^{6i}_0(a_1a_2a_0)^na_1 - T^{6i}_1b_1, \quad & T^{6i+2} = T^{6i+1}a_2  = T^{6i+1}b_2, \\
T^{6i+3}_j = \begin{cases} T^{6i+2}(a_0a_1a_2)^na_0 & \text{ if } j=0 \\  T^{6i+2}b_0 & \text{ if } j=1 \end{cases}, \quad  & T^{6i+4} = T^{6i+3}_0a_1 - T^{6i+3}_1b_1,\\
T^{6i+5} = T^{6i+4}(a_2a_1a_0)^na_2, \quad & T^{6i+6}_j = \begin{cases} T^{6i+5}a_0 & \text{ if } j=0 \\  T^{6i+5}b_0 & \text{ if } j=1 \end{cases}.
\end{align*}
\end{enumerate}

\begin{enumerate}
\item[(U)] \begin{align*}
U^{6i+1} = U^{6i}a_2 = U^{6i}b_2, \quad & U^{6i+2}_j = \begin{cases} U^{6i+1}(a_0a_1a_2)^na_0 & \text{ if } j=0 \\  U^{6i+1}b_0 & \text{ if } j=1 \end{cases} \\
U^{6i+3} = U^{6i+2}_0a_1 - U^{6i+2}_1b_1, \quad & U^{6i+4} = U^{6i+3}(a_2a_0a_1)^na_2, \\
U^{6i+5}_j = \begin{cases} U^{6i+4}a_0 & \text{ if } j=0 \\  U^{6i+4}b_0 & \text{ if } j=1 \end{cases}, \quad  & U^{6i+6} = U^{6i+5}_0(a_1a_2a_0)^na_1 - U^{6i+5}_1b_1.
\end{align*}
\end{enumerate}

For all elements $g\in\mathcal{G}^i$, the following are $\Lambda_n$-bimodules;
$$\mathbb{Q}_{m} = \bigoplus_{g\in\mathcal{G}^m} \Lambda_n \mfo(g)\ot \mft(g)\Lambda_n.$$
Note that for $g_1\neq g_2$ in $\mathcal{G}^m$, $\mfo(g_1)\ot \mft(g_1)\neq \mfo(g_2)\ot \mft(g_2)$ and a generating set of $\mathbb{Q}_m$ is $\{\mfo(g)\ot \mft(g)\}_{g\in\mathcal{G}^m} = \{e_r\ot e_s, f_r\ot e_s, e_r\ot f_s, f_r\ot f_s\}$ where $r,s =0,1,2$. For example, 
\begin{align*}
\mathbb{Q}_0 &=  \Lambda_n \mfo(R^0)\ot \mft(R^0)\Lambda_n \oplus  \Lambda_n \mfo(S^0)\ot \mft(S^0)\Lambda_n \\
& \oplus \Lambda_n \mfo(T^0)\ot \mft(T^0)\Lambda_n \oplus  \Lambda_n \mfo(U^0)\ot \mft(U^0)\Lambda_n  \\
& =  (\Lambda_n e_0\ot e_0\Lambda_n) \oplus  (\Lambda_n e_1\ot e_1\Lambda_n) \oplus (\Lambda_n f_1\ot f_1\Lambda_n) \oplus  (\Lambda_n e_2\ot e_2\Lambda_n),
\end{align*}
so a generating set for $\mathbb{Q}_0$ as a bimodule is $\{e_0\ot e_0, e_1\ot e_1, f_1\ot f_1, e_2\ot e_2\}$. Note that this consists of all paths of length zero from vertex $i$ to $i$. Similarly, $\mathbb{Q}_1$ will be generated by paths of length 1 from vertex $i$ to $j$, $\mathbb{Q}_2$ will be generated by paths of length 2 from vertex $i$ to $j$ and so on. No new $e_i\ot e_j$ is obtain when we consider paths of length greater than or equal four. Therefore, a basis for the generating set for $\mathbb{Q}_m$ occurs mode 3 in one of the following cases for $m\geq 0$.
\begin{itemize}
\item $\mathbb{Q}_{3m}$, generated by $\{e_0\ot e_0, e_1\ot e_1, e_1\ot f_1, f_1\ot e_1, f_1\ot f_1, e_2\ot e_2\}$,
\item $\mathbb{Q}_{3m+1}$, generated by $\{e_0\ot e_1, e_0\ot f_1, e_1\ot e_2, f_1\ot e_2, e_2\ot e_0\}$.
\item $\mathbb{Q}_{3m+2}$, generated by $\{e_0\ot e_2, e_1\ot e_0, f_1\ot e_0, e_2\ot e_1, e_2\ot f_1\}$,
\end{itemize}

Although basis elements are presented for $\mathbb{Q}_{3m}, \mathbb{Q}_{3m+1}$ and $\mathbb{Q}_{3m+2}$, the differentials occur in mode 6, so the chain complex has the following form;
$$\cdots \xrightarrow{}\mathbb{Q}_{3m+5}\xrightarrow{\partial^{6j+5} }\mathbb{Q}_{3m+4}\xrightarrow{\partial^{6j+4}}\mathbb{Q}_{3m+3}\xrightarrow{\partial^{6j+3}}\mathbb{Q}_{3m+2}=$$
$$=\mathbb{Q}_{3m+2}\xrightarrow{ \partial^{6j+2}}\mathbb{Q}_{3m+1}\xrightarrow{ \partial^{6j+1}}\mathbb{Q}_{3m}\xrightarrow{ \partial^{6j}}\mathbb{Q}_{3m-1}\xrightarrow{}\cdots$$
We now present a theorem that formally state these results.

\begin{theorem}\cite[Theorem 3.3]{furuya2012projective}\label{major-res}
Let $\Lambda_n = \kk\mathcal{Q}/I_n$ be a family of quiver algebras described in \eqref{family} and let $\mathcal{G}^i$ be a set of uniform relations described by \eqref{uniform-paths}. A minimal projective bimodule resolution of $\Lambda_n$  is given by
\begin{equation*}
\mathbb{Q}_{\bullet}(\Lambda_n):  \quad\cdots \rightarrow \mathbb{Q}_{3m}\xrightarrow{\;\partial^{6j}\; }\mathbb{Q}_{3m-1}\xrightarrow{\partial^{6j-5}} \cdots \xrightarrow{\;\partial^2\;}\mathbb{Q}_{1}\xrightarrow{\;\partial^1\;}\mathbb{Q}_{0}\;(\;\xrightarrow{\;\partial^0\;} \Lambda_n)
\end{equation*}
for $m\geq 0$, where the differentials $\partial^{\bullet}:\mathbb{Q}_{3m}\xrightarrow{}\mathbb{Q}_{3m-1}$ applied to the generators of each bimodule $\mathbb{Q}_{3m}$ are given in Definition \eqref{differentials} and $\partial^0$ is multiplication.
\end{theorem}

\begin{definition}\label{differentials}
\begin{enumerate}[label=(\roman*)]\ \\
\item For $j=0$, we take $\partial^{6j}=\partial^0$ and for $j\geq 1$, and $l=0,1,2$\begin{equation*}
    \partial^{6j}:\begin{cases}
    e_l\ot e_l \mapsto \begin{cases}
        e_0\ot (a_2a_0a_1)^na_2+ (\sum_{k=0}^{3n-1}a_0a_1a_2\dots a_k\ot a_k \dots a_{3n-1})\\
        +(a_0a_1a_2)^na_0\ot e_0-b_0\ot e_0 & \text{ if } l=0,\\ \\
        e_1\ot (a_0a_1a_2)^na_0+(\sum_{k=1}^{3n}a_1a_2a_3\dots a_k\ot a_k \dots a_{3n})\\
        +(a_1a_2a_0)^na_1\ot e_1 & \text{if } l=1,\\ \\
        e_2\ot (a_1a_2a_0)^na_1 + (\sum_{k=2}^{3n+1}a_2a_3a_4\dots a_k \ot a_k \dots a_{3n+1})
        \\ +(a_2a_0a_1)^na_2\ot e_2-e_2\ot b_1 & \text{ if } l=2,
    \end{cases}
    \\
    \\
    e_1\ot f_1 \mapsto e_1\ot b_0 +a_1\ot f_1,\\
    f_1\ot e_1 \mapsto f_1\ot a_0 +b_1\ot e_1,\\
    f_1\ot f_1\mapsto f_1\ot b_0 + b_1\ot f_1.
\end{cases}
\end{equation*}

\item For $j=0$, and $l=0,1,2$, $r=0,1$, we take 
$$\partial^{6j+1}=\partial^1: \begin{cases} e_l\ot e_{l+1} \mapsto e_l\ot a_l - a_l\ot e_{l+1}, \\ f_r\ot f_{r+1} \mapsto f_r\ot b_r - b_r\ot f_{r+1}, \end{cases}$$ 
and for $j\geq 1$ and $l=0,1,2$, $r=0,1$, we have
\begin{equation*}
    \partial^{6j+1}:\begin{cases}
        e_l\ot e_{l+1} \mapsto \begin{cases}
            e_0\ot a_0 -a_0\ot e_1 + b_0\ot e_1 & \text{ if } l=0,\\
            e_1\ot a_1- a_1\ot e_2 -e_1\ot b_1 & \text{ if }l=1,\\
            e_2\ot a_2 -a_2\ot e_0 & \text{ if }l=2,
        \end{cases}
        \\
        \\
        f_r\ot f_{r+1} \mapsto \begin{cases}
            e_0\ot b_0 + b_0 \ot f_1 -(a_0a_1a_2)^na_0\ot f_1 \text{ if } r=0,\\
            -f_1\ot b_1 -b_1\ot e_2 + f_1\ot (a_1a_2a_0)^n a_1 & \text{ if }r=1.
        \end{cases}
    \end{cases}
\end{equation*}

\item For $j\geq 0$, $l=0,1,2$ and $r=1,2$, we have
\begin{equation*}
\partial^{6j+2}: \begin{cases} e_l\ot e_{l+2} \mapsto \begin{cases} e_0\ot (a_1a_2a_0)^na_1 + \sum_{k=0}^{3n-1}a_0a_1a_2\cdots a_k\ot a_{k+2}\cdots a_{3n+1} \\
+(a_0a_1a_2)^na_0\ot e_2 - e_0\ot b_1 - b_0\ot e_2 & \text{ if } l=0,
\\
\\ e_l\ot (a_{l+1}a_{l+2}a_l)^na_{l+1} \\
+ \sum_{k=l}^{3n+l-1}a_la_{l+1}a_{l+2}\cdots a_k\ot a_{k+2}\cdots a_{3n+l+1} \\
+(a_la_{l+1}a_{l+2})^na_l\ot e_{l+2}  & \text{ if }l=1,2
\end{cases}
\\ 
\\
f_r\ot f_{r+2} \mapsto f_r\ot b_{r+1} + b_r\ot f_{r+2}, \end{cases}
\end{equation*}
\item For $j\geq 0$, $l=0,1,2$ we have
\begin{equation*}
\partial^{6j+3}: \begin{cases} e_l\ot e_{l} \mapsto \begin{cases} e_0\ot a_2 - a_0\ot e_0 + b_0\ot e_0 & \text{ if } l=0,\\
e_1\ot a_0 - a_1\ot e_1              & \text{ if } l=1,\\
e_2\ot a_1 - a_2\ot e_2 - e_2\ot b_1 & \text{ if } l=2,\\
\end{cases}
\\ 
\\
e_1\ot f_{1} \mapsto e_1\ot b_0 - (a_1a_2a_0)^n a_1\ot f_{1}, \\
f_1\ot e_{1} \mapsto f_1\ot (a_0a_1a_2)^n a_0 -  b_1\ot e_{1}, \\
f_1\ot f_{1} \mapsto f_1\ot b_{0} - b_1\ot f_{1}. 
\end{cases}
\end{equation*}
\item For $j\geq 0$, $l=0,1,2$ and $r=0,1$
\begin{equation*}
    \partial^{6j+4}: \begin{cases} e_l\ot e_{l+1} \mapsto \begin{cases}
        e_0\ot (a_0a_1a_2)^na_0 +(\sum_{k=0}^{3n-1}a_0a_1a_2\dots a_k\ot a_{k+1}\dots a_{3n})\\+(a_0a_1a_2)^na_0\ot e_1-b_0\ot e_1 & \text{ if } l=0,\\ \\
        e_1\ot (a_1a_2a_0)^na_1 + (\sum_{k=1}^{3n}a_1a_2a_3 \dots a_k\ot a_{k+1}\dots a_{3n+1}) \\
        +(a_1a_2a_0)^na_1\ot e_2-e_1\ot b_1 & \text{ if } l=1,\\ \\
        e_2\ot (a_2a_0a_1)^na_2+\\
        (\sum_{k=2}^{3n+1}a_2a_3a_4\dots a_k \ot a_{k+1}\dots a_{3n+2}) + (a_2a_0a_1)^na_2\ot e_0 & \text{ if } l=2,\\
    \end{cases}
    \\
    \\
        f_r\ot f_{r+1} \mapsto \begin{cases}
            e_0\ot b_0 - b_0\ot f_1 + a_0\ot f_1 & \text{ if } r=0,\\
            -f_1\ot b_1 + b_1\ot e_2 + f_1\ot a_1 & \text{ if } r=1.
        \end{cases}
    \end{cases}
\end{equation*}
\item For $j\geq 0$, $l=0,1,2$ and $r=1,2$
\begin{equation*}
    \partial^{6j+5} : \begin{cases}
        e_l\ot e_{l+2} \mapsto \begin{cases}
            e_0\ot a_1 -a_0\ot e_2 -e_0\ot b_1 + b_0\ot e_2 & \text{ if } l=0,\\
            e_1\ot a_2-a_1\ot e_0 & \text{ if }l=1,\\
            e_2\ot a_0 - a_2\ot e_1 & \text{ if }l=2,\\
        \end{cases}
        \\
        \\
        f_r\ot f_{r+2} \mapsto \begin{cases}
            f_1\ot (a_2a_0a_1)^na_2-b_1\ot e_0 & \text{ if } r=1,\\
            e_2\ot b_0 -(a_2a_0a_1)^na_2\ot f_1 & \text{ if } r=2.
        \end{cases}
    \end{cases} 
\end{equation*}
\end{enumerate}
\end{definition}\ \\

\textbf{Module structure compatibility:} Suppose that $A$ and $B$ are $\kk$-algebras, $M$ is a left $A$-module and $N$ is a left $B$-module. Then the tensor product $M\ot_{A\ot B} N$ is a left $A\ot B$ module and we have the following module structure
$(A\ot B)\times (M\ot_{A\ot B} N) \rightarrow M\ot_{A\ot B} N$ given by $(a\ot b)\cdot (m\ot n) \mapsto am\ot bn$ for every $a\in A, b\in B, m\in M$ and $n\in N.$

Let $A\ot A^{op}\cong A^e$ be the enveloping algebra of $A$ where $(A^{op},\cdot_{op})$ is the opposite algebra of $A$ with opposite multiplication given by $a\cdot_{op}b =ba$ for every $a,b\in A$. If $M$ is a left $A^e$-module, then we have the following module structure $A^e\times M\rightarrow M$ given by $(a\ot a')\cdot m = ama'$. If $N$ is also a left $A^e$-module, then $M\ot_{A^e} N := M\ott N$ is a left $A^e\ot A^e$-module and the module structure $(A^e\ot A^e)\times (M\ot_{A^e} N) \rightarrow M\ot_{A^e} N$ is given by $(r\ot s)\cdot (m\ott n) = r\cdot m\ott s\cdot n$ for every $r, s\in A^e, m\in M$ and $n\in N$. Since $A^e\ot A^e \cong A^e$, then $M\ot_{A^e} N$ is also a left $A^e$-module and the module structure $A^e\times (M\ot_{A^e} N) \rightarrow M\ot_{A^e} N$ is given by $(r\ot r')\cdot (m\ott n) = rm\ott nr'$ for every $r, r'\in A, m\in M$ and $n\in N$. Since these module structures are compatible for these algebras, we have in particular, for $r_1, r_2, s_1, s_2\in A$, 
\begin{align}\label{mod1}
(r_1\ot s_1)(r_2\ot s_2)\cdot (m\ott n) = (r_1r_2\ot s_2s_1)\cdot (m\ott n) = r_1r_2m \ott ns_2s_1,
\end{align}
and when we use the module structure for $(A^e\ot A^e)\times (M\ot_{A^e} N) \rightarrow M\ot_{A^e} N$ earlier given, we obtain
\begin{align}\label{mod2}
(r_1\ot s_1)(r_2\ot s_2)\cdot (m\ott n) = ((r_1\ot s_1)\cdot m) \ott ((r_2\ot s_2) \cdot n) = r_1m s_1 \ott r_2ns_2
\end{align}
since both $M$ and $N$ are left $A^e$-modules.

\begin{remark}
Let $\Lambda_n$ be the family of quiver algebras in \eqref{family} and let $\lambda_1,\lambda_2, \beta_1, \beta_2\in \Lambda_n$. Let $\mathbb{Q}_a$ and $\mathbb{Q}_b$ be left $\Lambda_n^e$-modules having basis elements $e_i\ot e_j$ and $e_l\ot e_s$ respectively. Using the module structure of \eqref{mod1} and \eqref{mod2}, we have the following module structure compatibility.

\begin{equation}\label{modulestructure}
    \lambda_1(e_i\ot e_j)\beta_1 \ott\lambda_2(e_l\ot e_s) \beta_2 \simeq \lambda_1\lambda_2(e_i\ot e_j)\ott(e_l\ot e_s)\beta_2\beta_1.
\end{equation}
\end{remark}

\section{Main results}\label{Main results}
We denote the total complex obtained from the tensor product complex $\mathbb{Q}_{\bullet}\ot_{\Lambda_n}\mathbb{Q}_{\bullet}$ by $\mathbb{P}_{\bullet}$, so that $\Delta:\mathbb{Q}_{\bullet}\xrightarrow{}\mathbb{P}_{\bullet}$ is the comultiplicative map lifting the identity map $\Lambda_n\cong\Lambda_n\ott\Lambda_n$. We define $\displaystyle{\mathbb{P}_{m}: =\bigoplus_{a+b=m} \mathbb{Q}_a\ott \mathbb{Q}_b},$
and for $x\ot y\in \mathbb{Q}_a\ott \mathbb{Q}_b \in \mathbb{P}_m$, the differential $d^m:\mathbb{P}_m\rightarrow \mathbb{P}_{m-1}$ is given by
$$d^m(x\ot y)= \partial^{m-1}(x)\ot y + (-1)^{a} x\ot \partial^{m-1}(y)$$
provided $\partial^{m-1}(x)$ or  $\partial^{m-1}(y)$ are non-zero. We seek maps $\Delta_{\bullet}$ for which the following diagrams are commutative. As previously noted, the generating set for the free modules $\mathbb{Q}_{\bullet}$ are the same for modules in degrees $3m, 3m+1$ and $3m+2$ respectively. 
$$\begin{tikzcd}
\mathbb{Q}_{\bullet}: =\arrow{d}{\Delta_{\bullet}} \cdots \arrow{r}{}
    & \mathbb{Q}_{3m+2}\arrow{d}{\Delta_{3m+2}}\arrow{r}{\partial^{6j+2}}
        & \mathbb{Q}_{3m+1} \arrow{d}{\Delta_{3m+1}} \arrow{r}{\partial^{6j+1}}
            & \mathbb{Q}_{3m} \arrow{d}{\Delta_{3m}} \arrow{r}{\partial^{6j}}
                & \cdots \\
\mathbb{P}_{\bullet}:= \cdots \arrow{r}{d^*}
    & \mathbb{P}_{3m+2} \arrow{r}{d^*}
        & \mathbb{P}_{3m+1} \arrow{r}{d^*}
            & \mathbb{P}_{3m} \arrow{r}{d^*}
                & \cdots
\end{tikzcd}$$

Moreso, the differentials occur in mode 6, so the chain complex has the following forms;
$$\cdots \xrightarrow{}\mathbb{Q}_{3m+2}\xrightarrow{\partial^{6j+5}, \partial^{6j+2}}\mathbb{Q}_{3m+1}\xrightarrow{\partial^{6j+4}, \partial^{6j+1}}\mathbb{Q}_{3m}\xrightarrow{\partial^{6j+3}, \partial^{6j}}\mathbb{Q}_{3m-1}\xrightarrow{}\cdots$$
but it is sufficient to define the comultiplicative map in degrees $3m, 3m+1$ and $3m+2$.

In the following Lemma, we define a chain map $\Delta'$ lifting the bimodule homomorphism $\pi:\Lambda_n\rightarrow\Lambda_n\ott\Lambda_n$ given by $x\mapsto 2(x\ot x)$ for every idempotent $x\in\Lambda_n$. Such a chain map exists by fundamental results of homological algebra and the following diagram is commutative in each degree $m$, i.e. $\Delta^{'}_{n,m}\partial^{*}= d^{*}\Delta^{'}_{n,m+1}$
for every algebra $\Lambda_n$.
$$\begin{tikzcd}
\mathbb{Q}_{\bullet}\arrow{d}{\Delta_{n,\bullet}^{'}} \arrow{r}{\partial^{*}}
    & \mathbb{Q}_{\bullet} \arrow{d}{\Delta_{n,\bullet}^{'}}  \arrow{r}{} & \Lambda_n \arrow{d}{\pi}\\
\mathbb{P}_{\bullet} \arrow{r}{d^{*}}
    & \mathbb{P}_{\bullet} \arrow{r}{} & \Lambda_n \ott\Lambda_n\cong \Lambda_n
\end{tikzcd}$$

\begin{lemma}\label{deltap-chain_map}
Let $\Lambda_n = \kk Q/I_n$, ($char(\kk)\neq 2$) be the family of quiver algebras of \eqref{family} and let $\mathcal{G}^i$ be a set of uniform relations described by uniform-paths. Let $w,w_i,w_j$ be any idempotent in the set $\{e_0,e_1,f_1,e_2\}.$ A chain map 
$\displaystyle{\Delta^{'}_{n,\bullet}: \mathbb{Q}_{\bullet} \rightarrow \mathbb{Q}_{\bullet}\ott \mathbb{Q}_{\bullet}},$ lifting the map $\pi:\Lambda_n\rightarrow\Lambda_n$ given by $x\mapsto 2x$ for every $x\in\Lambda_n$ can be given by
\begin{align*}
\Delta^{'}_{n,\ast}(w_i\ot w_j) &= (w_i\ot w_i) \bar{\ot} (w_i\ot w_j) + (w_i\ot w_j) \bar{\ot} (w_j\ot w_j) \\
& + (w_i\ot w_j) \bar{\ot} (w_i\ot w_i) + (w_j\ot w_j) \bar{\ot} (w_i\ot w_j) \\
& + \sum_{w\neq w_i,w_j} (w\ot w) \bar{\ot} (w_i\ot w_j) + (w_i\ot w_j) \bar{\ot} (w\ot w) \qquad \text{and}\\
\Delta^{'}_{n,\ast}(w_i\ot w_i) &= (w_i\ot w_i) \bar{\ot} (w_i\ot w_i) + (w_i\ot w_i) \bar{\ot} (w_i\ot w_i) \\
& + \sum_{w\neq w_i} (w\ot w) \bar{\ot} (w_i\ot w_i) + (w_i\ot w_i) \bar{\ot} (w\ot w)
\end{align*}
where $\bar{\ot} =\ot_{\Lambda_n}$.
\end{lemma}

\begin{proof}
Since the generating set for modules in the resolution occurs periodically, we will prove that the result holds in the case $m=0$. That is, we show commutativity of the first, second and third squares in the following diagram. $$\begin{tikzcd}
 \cdots \arrow{r}{\partial^{3}}
    & \mathbb{Q}_{2}\arrow{d}{\Delta^{'}_{n,2}}\arrow{r}{\partial^{2}}
        & \mathbb{Q}_{1} \arrow{d}{\Delta^{'}_{n,1}} \arrow{r}{\partial^{1}}
         & \mathbb{Q}_{0} \arrow{d}{\Delta^{'}_{n,0}} \arrow{r}{\partial^{0}}
            & \Lambda_n \arrow{d}{\pi} \\
\cdots \arrow{r}{d^*}
    & (\mathbb{Q}_2\ott \mathbb{Q}_0)\oplus(\mathbb{Q}_1\ott\mathbb{Q}_1) \oplus(\mathbb{Q}_0\ott\mathbb{Q}_2) \arrow{r}{d^2}
        & \stackrel{\stackrel{(\mathbb{Q}_1\ott \mathbb{Q}_0)}{\oplus}}{ (\mathbb{Q}_0\ott \mathbb{Q}_1)} \arrow{r}{d^1}
            & (\mathbb{Q}_0\ott\mathbb{Q}_0) \arrow{r}{\partial^0\ot\partial^0}
               &\Lambda_n\ott \Lambda_n &
\end{tikzcd}$$
To prove that the diagram is commutative in the first three squares, we will show that (i) $\pi\partial^{0} = (\partial^0\ot\partial^0)\Delta^{'}_{n,0}$, (ii) $\Delta^{'}_{n,0}\partial^{1} = d^{1}\Delta^{'}_{n,1}$ and (iii) $\Delta^{'}_{n,1}\partial^{2} = d^{2}\Delta^{'}_{n,2}.$ \\
\noindent \textbf{Case (i):}  We now show for basis element $e_0\ot e_0$ that $\pi\partial^{0}(e_0\ot e_0) = (\partial^0\ot\partial^0)\Delta^{'}_{n,0}(e_0\ot e_0)$ and a similar proof can be applied to other basis elements of $\mathbb{Q}_0.$ The left hand side of the expression becomes $\pi\partial^{0}(e_0\ot e_0) = \pi(e_0^2)= \pi(e_0)= 2(e_0\ot e_0)=2e_0$ after applying multiplication. Since $e_0$ is an idempotent. The right hand side becomes
\begin{align*}
&(\partial^0\ot\partial^0)\Delta^{'}_{n,0}(e_0\ot e_0) = \partial^0\ot\partial^0 \big(e_0\ot e_0\bar{\ot}e_0\ot e_0 + e_0\ot e_0\bar{\ot}e_0\ot e_0 \\
    &+ e_1\ot e_1\bar{\ot}e_0\ot e_0 + e_0\ot e_0\bar{\ot}e_1\ot e_1 + e_2\ot e_2 \bar{\ot}e_0\ot e_0 + e_0\ot e_0 \bar{\ot} e_2\ot e_2 \\
    &+ f_1\ot f_1\bar{\ot}e_0 \ot e_0 + e_0\ot e_0 \bar{\ot}f_1\ot f_1\big)\\
    &=e_0\ott e_0 + e_0\ott e_0 + e_1\ott e_0 + e_0\ott e_1 + e_2\ott e_0 + e_0\ott e_2 +f_1\ott e_0 + e_0\ott f_1\\
    &= e_0 + e_0 =2e_0 = \pi\partial^{0}(e_0\ot e_0).
\end{align*}
\noindent \textbf{Case (ii):}
We will show that the diagram is commutative for the case $\partial^* = \partial^{6j+1}$ where $j=0$, and the computation is similar for $j\geq 0$. We show for basis element $e_0\ot e_1$ that $\Delta^{'}_{n,0}\partial^{1}(e_0\ot e_1) = d^{1}\Delta^{'}_{n,1}(e_0\ot e_1)$
and a similar proof can be applied to other basis elements of $\mathbb{Q}_1.$  Using $d^1=\partial\ot 1 +(-1)^\ast 1\ot \partial$, we have
\begin{multline*}
d^{1}\Delta^{'}_{n,1}(e_0\ot e_1)=d^{1}(e_2\ot e_2 \ott e_0\ot e_1 +e_0\ot e_1\ott e_2\ot e_2+ e_0\ot e_0\ott e_0\ot e_1+ e_0\ot e_1\ott e_0\ot e_0 \\
+ e_0\ot e_1\ott e_1\ot e_1 + e_1\ot e_1\ott e_0\ot e_1+ e_0\ot e_1\ott f_1 \ot f_1+ f_1\ot f_1\ott e_0\ot e_1 )\\
=(\partial\ot 1 +(-1)^\ast 1\ot \partial) [e_2\ot e_2\ott e_0\ot e_1 + e_0\ot e_1 \ott e_2\ot e_2+ \\ e_0\ot e_0\ott e_0\ot e_1+ e_0\ot e_1\ott e_0\ot e_0+ 
e_0\ot e_1\ott e_1\ot e_1+ e_1\ot e_1\ott e_0\ot e_1 \\
+ e_0\ot e_1\ott f_1 \ot f_1+ f_1\ot f_1\ott e_0\ot e_1]\\=
\partial(e_2\ot e_2)\ott e_0\ot e_1 + e_2\ot e_2\ott \partial (e_0\ot e_1)+\\
\partial(e_0\ot e_1)\ott e_2\ot e_2 -e_0\ot e_1\ott \partial(e_2\ot e_2)+ \\\partial(e_0\ot e_0)\ott e_0\ot e_1+e_0\ot e_0\ott \partial(e_0\ot e_1)+\\\partial(e_0\ot e_1)\ott e_0\ot e_0 -e_0\ot e_1\ott \partial(e_0\ot e_0) +
\\ \partial(e_0\ot e_1)\ott e_1\ot e_1 -e_0\ot e_1\ott \partial(e_1\ot e_1) +\\ \partial(e_1\ot e_1)\ot e_0\ot e_1 +e_1\ot e_1\ott \partial(e_0\ot e_1)+\\ \partial(e_0\ot e_1)\ott f_1\ot f_1 -e_0\ot e_1\ott \partial(f_1\ot f_1)+\\ \partial(f_1\ot f_1)\ott e_0\ot e_1 + f_1\ot f_1\ott \partial(e_0\ot e_1)
\end{multline*}
\begin{multline*}
=0+ e_2\ot e_2\ott[e_0\ot a_0 -a_0\ot e_1]+ [e_0\ot a_0 -a_0 \ot e_1]\ott e_2\ot e_2 - 0 + 0+ e_0\ot e_0\ott [e_0\ot a_0 - a_0\ot e_1] +\\
[e_0\ot a_0 -a_0\ot e_1]\ott e_0\ot e_0- 0+\\
[e_0\ot a_0 -a_0\ot e_1]\ott e_1\ot e_1+0+\\
0+e_1\ot e_1\ott [e_0\ot a_0 -a_0\ot e_1]+\\
[e_0\ot a_0 -a_0\ot e_1]\ott f_1\ot f_1-0+\\
0+f_1\ot f_1\ott [e_0\ot a_0 -a_0\ot e_1]
\end{multline*}
\begin{multline*}
=(e_2\ot e_2\ott e_0\ot e_0 + e_0\ot e_0\ott e_2\ot e_2 + e_0\ot e_0 \ott e_0\ot e_0+e_0\ot e_0\ott e_0\ot e_0 + e_0\ot e_0\ott e_1\ot e_1\\
+ e_1\ot e_1\ott e_0\ot e_0+ e_0\ot e_0\ott f_1\ot f_1 + f_1\ot f_1\ott e_0\ot e_0) a_0-a_0(e_2\ot e_2\ott e_1\ot e_1 +e_1\ot e_1\ott e_2\ot e_2 \\+e_0\ot e_0\ott e_1\ot e_1+e_1\ot e_1\ott e_0\ot e_0 +e_1\ot e_1\ott e_1\ot e_1+ \\e_1\ot e_1\ott e_1\ot e_1+e_1\ot e_1\ott f_1\ot f_1 + f_1\ot f_1\ott e_1\ot e_1)
\end{multline*}
On the other hand, we have
\begin{multline*}
    \Delta^{'}_{n,0}\partial^1(e_0\ot e_1)= \Delta^{'}_{n,0}(e_0\ot a_0-a_0\ot e_1)=\Delta^{'}_{n,0}(e_0\ot e_0)a_0-a_0\Delta^{'}_{n,0}(e_1\ot e_1)\\
    =(e_0\ot e_0\bar{\ot} e_0\ot e_0 + e_0\ot e_0\bar{\ot} e_0\ot e_0+e_1\ot e_1\bar{\ot}e_0\ot e_0+e_0\ot e_0\bar{\ot}e_1\ot e_1 +\\ e_2\ot e_2 \bar{\ot}e_0\ot e_0+ e_0\ot e_0 \bar{\ot} e_2\ot e_2 + f_1\ot f_1\bar{\ot}e_0 \ot e_0 + e_0\ot e_0 \bar{\ot}f_1\ot f_1)a_0 \\
    -a_0(e_1\ot e_1 \bar{\ot} e_1\ot e_1 + e_1\ot e_1 \bar{\ot} e_1\ot e_1 + e_0\ot e_0\bar{\ot}e_1\ot e_1 + e_1\ot e_1\bar{\ot}e_0\ot e_0 + e_2\ot e_2 \bar{\ot}e_1\ot e_1 \\
    +e_1\ot e_1\bar{\ot}e_2\ot e_2+ f_1\ot f_1\bar{\ot}e_1\ot e_1 +e_1\ot e_1\bar{\ot}f_1\ot f_1) = d^{1}\Delta^{'}_{n,1}(e_0\ot e_1)
\end{multline*}

\noindent\textbf{Case (iii):} We now show for basis element $e_0\ot e_2$ that $$\Delta^{'}_{n,1}\partial^{6j+2}(e_0\ot e_2) = d^{*}\Delta^{'}_{n,2}(e_0\ot e_2)$$ and a similar proof can be applied to other basis elements of $\mathbb{Q}_2.$ Starting with the left hand side and using the definition of $\partial^{6j+2}$, we have
\begin{multline*}
\Delta^{'}_{n,1}\partial^{6j+2}(e_0\ot e_2) = \Delta^{'}_{n,1}\Big(e_0\ot e_1(a_1a_2a_0)^na_1+\sum_{k=0}^{3n-1}a_0a_1a_2\dots a_k\ot a_{k+2}\dots a_{3n+1} \\+(a_0a_1a_2)^na_0e_1\ot e_2-e_0\ot f_1 b_1-b_0f_1\ot e_2\Big)
\end{multline*}
and expanding the summation gives
\begin{multline*}
    =\Delta^{'}_{n,1} \Big ( e_0\ot e_1 (a_1a_2a_0)^na_1+ a_0e_1\ot e_2 a_2 a_3\dots a_{3n+1} + a_0a_1e_2\ot e_0 a_0 a_4\dots a_{3n+1}+ \\ a_0a_1a_2 e_0\ot e_1 a_1 a_5\dots a_{3n+1} + a_0a_1a_2a_0 e_1\ot e_2 a_2 a_6\dots a_{3n+1}+ \dot a_0a_1\dots a_{3n-1}\ot a_{3n+1} + \\ (a_0a_1a_2)^na_0 e_1\ot e_2 -e_0\ot f_1 b_1-b_0 f_1\ot e_2 \Big )\\
    =\Delta^{'}_{n,1}\Big ( 1 (e_0\ot e_1)(a_1a_2a_0)^na_1 +\sum_{t=1}^n (a_0a_1a_2)^t (e_0\ot e_1)(a_1a_2a_0)^{n-t}a_1\\ + a_0(e_1\ot e_2)(a_2a_0a_1)^n + \sum_{t=1}^n a_0(a_1a_2a_0)^t e_1\ot e_2 (a_2a_0a_1)^{n-t} + a_0a_1(e_2\ot e_0)(a_0a_1a_2)^{n-1}a_0a_1 \\ +\sum_{t=1}^{n}a_0a_1(a_2a_0a_1)^te_e\ot e_0 (a_0a_1a_2)^{n-t-1}a_0a_1-1(e_0\ot f_1)b_1-b_0(f_1\ot e_2)1\Big).
\end{multline*}
After using the definition $\Delta^{'}_{n,1}(x(e_i\ot e_j)y)=x(\Delta^{'}_{n,1}(e_i\ot e_j))y$ for every $x,y\in\Lambda_n$, we get
\begin{multline}\label{2.8}
\sum_{t=0}^n (a_0a_1a_2)^t\Big [\stackrel{(1)}{e_0\ot e_0 \bar{\ot} e_0\ot e_1} + \stackrel{(7)}{e_0\ot e_1\bar{\ot}e_1\ot e_1} + \stackrel{(6)}{e_2\ot e_2 \bar{\ot} e_0\ot e_1}\color{black} \\+ \stackrel{(4)}{f_1\ot f_1 \bar{\ot}e_0\ot e_1}\color{black} +  \stackrel{(5)}{e_0\ot e_1 \bar{\ot}e_0\ot e_0}\color{black} + \stackrel{(3)}{e_1\ot e_1\bar{\ot}e_0\ot e_1}\color{black}\\+ \stackrel{(2)}{e_0\ot e_1\bar{\ot}e_2\ot e_2} \color{black}+ \stackrel{(8)}{e_0\ot e_1\bar{\ot}f_1\ot f_1}\Big ] (a_1a_2a_0)^{n-t}a_1
\end{multline}
$$+$$
\begin{multline}\label{2.9}
    \sum_{t=0}^n a_0(a_1a_2a_0)^t \Big [ \stackrel{(3)}{e_1\ot e_1 \bar{\ot} e_1\ot e_2} \color{black}+ \stackrel{(2)}{e_1\ot e_2 \bar{\ot} e_2\ot e_2} \color{black}+ \stackrel{(1)}{e_0\ot e_0\bar{\ot} e_1\ot e_2} \color{black}+\\ \stackrel{(4)}{f_1\ot f_1\bar{\ot} e_1\ot e_2} \color{black} + \stackrel{(7)}{e_1\ot e_2 \bar{\ot} e_1\ot e_1} \color{black} + \stackrel{(6)}{e_2\ot e_2 \ot e_1\ot e_2}\color{black} \\+ \stackrel{(5)}{e_1\ot e_2 \bar{\ot} e_0\ot e_0} \color{black}+ \stackrel{(8)}{e_1\ot e_2 \bar{\ot} f_1\ot f_1} \Big ] (a_2a_0a_1)^{n-t}
\end{multline}
$$+$$
\begin{multline}\label{2.10}
    \sum_{t=0}^{n}a_0a_1(a_2a_0a_1)^t \Big [ \stackrel{(6)}{e_2\ot e_2 \bar{\ot} e_2\ot e_0} \color{black} + \stackrel{(5)}{e_2\ot e_0 \bar{\ot} e_0\ot e_0}\color{black} + \stackrel{(3)}{e_1\ot e_1 \bar{\ot} e_2\ot e_0} \color{black} +\\ \stackrel{(4)}{f_1\ot f_1\bar{\ot} e_2\ot e_0} \color{black}  + \stackrel{(2)}{e_2\ot e_0 \bar{\ot} e_2\ot e_2} \color{black}+ \stackrel{(1)}{e_0\ot e_0 \bar{\ot} e_2\ot e_0}\color{black}\\+ \stackrel{(7)}{e_2\ot e_0 \bar{\ot} e_1\ot e_1}\color{black} + \stackrel{(8)}{e_2\ot e_0\bar{\ot} f_1\ot f_1}\Big ](a_0a_1a_2)^{n-t-1}a_0a_1
\end{multline}
$$+$$
\begin{multline}\label{2.11}
    -\Big [\stackrel{(1)}{e_0\ot e_0 \bar{\ot} e_0\ot f_1} \color{black} +\stackrel{(8)}{e_0\ot f_1 \bar{\ot} f_1\ot f_1} + \stackrel{(3)}{e_1\ot e_1 \bar{\ot}e_0\ot f_1} \color{black} \\+ \stackrel{(6)}{e_2\ot e_2\bar{\ot} e_0\ot f_1}\color{black} + \stackrel{(5)}{e_0\ot f_1\bar{\ot} e_0\ot e_0} \color{black}\\ + \stackrel{(4)}{f_1 \ot f_1 \bar{\ot} e_0\ot f_1} \color{black}+ \stackrel{(7)}{e_0\ot f_1 \bar{\ot} e_1\ot e_1}\color{black} + \stackrel{(2)}{e_0\ot f_1 \bar{\ot} e_2\ot e_2}\color{black}\Big ] b_1
    \end{multline}
    \begin{center}
        $+$
    \end{center}
    \begin{multline}\label{2.12}
    -b_0 \Big [ \stackrel{(4)}{f_1\ot f_1 \bar{\ot} f_1\ot e_2} \color{black}+ \stackrel{(2)}{f_1 \ot e_2 \bar{\ot} e_2\ot e_2} \color{black}+ \stackrel{(1)}{e_0 \ot e_0 \bar{\ot}f_1\ot e_2} \color{black} \\+ \stackrel{(3)}{e_1\ot e_1 \bar{\ot} f_1 \ot e_2}\color{black} +\stackrel{(8)}{f_1\ot e_2 \bar{\ot} f_1\ot f_1} \\ + \stackrel{(6)}{e_2\ot e_2 \bar{\ot} f_1 \ot e_2} \color{black}+ \stackrel{(5)}{f_1 \ot e_2 \bar{\ot} e_0\ot e_0} \color{black}+ \stackrel{(7)}{f_1\ot e_2 \bar{\ot} e_1 \ot e_1} \color{black}\Big]
\end{multline}

From Equations \eqref{2.8}, \eqref{2.9}, \eqref{2.10}, \eqref{2.11}, and \eqref{2.12}, all expressions labelled $\stackrel{(1)}{}$ \color{black} yields
\begin{multline}
    =\sum_{t=0}^{n} (a_0a_1a_2)^t (e_0\ot e_0 \ott e_0 \ot e_1) (a_1a_2a_0)^{n-t} a_1 \\+ \sum_{t=0}^{n} a_0 (a_1a_2a_0)^t (e_0\ot e_0 \bar{\ot} e_1\ot e_2) (a_2a_0a_1)^{n-t} \\+ \sum_{t=0}^n a_0a_1 (a_2a_0a_1)^t(e_0\ot e_0 \bar{\ot} e_2\ot e_0)(a_0a_1a_2)^{n-t-1}a_0a_1\\-e_0\ot e_0\bar{\ot}e_0\ot f_1 b_1 - b_0 e_0\ot e_0 \bar{\ot} f_1 \ot e_2.
\end{multline}

Applying the module structure as defined in Equation \eqref{modulestructure}, we obtain

\begin{multline}
    =e_0\ot e_0 \bar{\ot} \Big (\sum_{t=0}^n (a_0a_1a_2)^te_0\ot e_1 (a_1 a_2 a_0)^{n-t}a_1  \Big ) \\
    + e_0\ot e_0 \bar{\ot} \Big ( \sum_{t=0}^n a_0(a_1a_2a_0)^t e_1\ot e_2 (a_2 a_0 a_1)^{n-t} \Big) \\ +e_0\ot e_0 \bar{\ot} \Big (\sum_{t=0}^n a_0a_1 (a_2a_oa_1)^te_2\ot e_0 (a_0a_1a_2)^{n-t-1}a_0a_1\Big )-e_0\ot e_0 \bar{\ot} \Big ( e_0\ot f_1 b_1 \Big )\\ -e_0\ot e_0 \bar{\ot} \Big (b_0 f_1\ot e_2\Big )
\end{multline}

\begin{multline}
    =e_0\ot e_0 \bar{\ot} \Big [ \Big ( \sum_{t=0}^n (a_0a_1a_2)^t e_0\ot e_1 (a_1a_2 a_0)^{n-t}a_1 \Big ) \\
    + \Big ( \sum_{t=0}^n a_0(a_1a_2a_0)^te_1\ot e_2 (a_2 a_0 a_1)^{n-t}\Big ) + \Big ( \sum_{t=0}^{n} a_0a_1(a_2a_0a_1)^t e_2\ot e_0(a_0a_1a_2)^{n-t-1}a_0a_1 \Big )\\ -\Big ( e_0\ot f_1 b_1 \Big ) -\Big( b_0 f_1\ot e_2 \Big ) \Big ]
\end{multline}

\begin{equation}
    =e_0\ot e_0\bar{\ot} \partial^{6j+2}(e_0\ot e_2)
\end{equation}

    \noindent Similarly, all expressions labelled $\stackrel{(2)}{}$ \color{black} yields $\partial^{6j+2}(e_0\ot e_2)\bar{\ot}e_2\ot e_2$.

    \noindent All expression labelled $\stackrel{(3)}{}$ \color{black} yields $e_1\ot e_1 \bar{\ot}\partial^{6j+2}(e_0\ot e_2)$.

    \noindent All expression labelled $\stackrel{(4)}{}$ \color{black} yields $f_1\ot f_1 \bar{\ot} \partial^{6j+2}(e_0\ot e_2)$.

    \noindent All expression labelled $\stackrel{(5)}{}$ \color{black} yields $\partial^{6j+2}(e_0\ot e_2)\bar{\ot}(e_0\ot e_0).$

    \noindent All expression labelled $\stackrel{(6)}{}$ \color{black} yields $e_2\ot e_2\bar{\ot}\partial^{6j+2}(e_0\ot e_2).$

    \noindent All expression labelled $\stackrel{(7)}{}$ \color{black} yields $\partial^{6j+2}(e_0\ot e_2)\bar{\ot}e_1\ot e_1$.

    \noindent and all expressions labelled $\stackrel{(8)}{}$ yields $\partial^{6j+2}(e_0\ot e_2)\bar{\ot}f_1\ot f_1$.
\medskip
Finally, $\Delta_{n,1}\partial^{6j+2}(e_0\ot e_2)$ is equal to
\begin{multline*}
e_0\ot e_0 \bar{\ot} (\partial^{6j+2}(e_0\ot e_2))+ \partial^{6j+2}(e_0\ot e_2)\bar{\ot}e_2\ot e_2 + e_1\ot e_1 \ott \partial^{6j+2}(e_0\ot e_2) + \\ f_1\ot f_1 \bar{\ot} \partial^{6j+2}(e_0\ot e_2)+ \partial^{6j+2}(e_0\ot e_2)\bar{\ot} e_0\ot e_0 +  e_2\ot e_2 \bar{\ot}\partial^{6j+2}(e_0\ot e_2) \\
+ \partial^{6j+2}(e_0\ot e_2)\bar{\ot} e_1\ot e_1 +\partial^{6j+2}(e_0\ot e_2)\bar{\ot} f_1\ot f_1.
\end{multline*}

On the other hand, \begin{multline}
    d^{*}\Delta'_{n,2}(e_0\ot e_2)= d^{6j+2}[e_0\ot e_0 \bar{\ot}e_0\ot e_2+ e_0\ot e_2 \bar{\ot}e_2\ot e_2 + e_1\ot e_1 \bar{\ot}e_0\ot e_2 \\+ f_1\ot f_1\bar{\ot}e_0\ot e_2 + e_0\ot e_2\bar{\ot} e_0\ot e_0 + e_2\ot e_2 \bar{\ot} e_0\ot e_2 \\+ e_0\ot e_2 \ott e_1\ot e_1 + e_0\ot e_2 \bar{\ot} f_1\ot f_1],
\end{multline}
\medskip
using $d^{*}=\partial \ot 1 + (-1)^\ast 1\ot \partial$, we would obtain \begin{multline}
e_0\ot e_0 \bar{\ot} (\partial^{6j+2}(e_0\ot e_2))+ \partial^{6j+2}(e_0\ot e_2)\bar{\ot}e_2\ot e_2 + e_1\ot e_1 \ott \partial^{6j+2}(e_0\ot e_2) + \\ f_1\ot f_1 \bar{\ot} \partial^{6j+2}(e_0\ot e_2)+ \partial^{6j+2}(e_0\ot e_2)\bar{\ot} e_0\ot e_0 + \\ e_2\ot e_2 \bar{\ot}\partial^{6j+2}(e_0\ot e_2)+ \partial^{6j+2}(e_0\ot e_2)\bar{\ot} e_1\ot e_1 +\partial^{6j+2}(e_0\ot e_2)\bar{\ot} f_1\ot f_1.
\end{multline}
Therefore, for all natural number $n$, we have 
\begin{equation}
    \Delta^{'}_{n,1}\partial^{6j+2}(e_0\ot e_2)= d^{6j+2}\Delta^{'}_{n,2}(e_0\ot e_2).
\end{equation}
Since it is straightforward to make similar calculations for subsequent squares on all basis elements $e_i\ot e_j$, we conclude that 
$$\Delta^{'}_{n,m}\partial^{*}(e_i\ot e_j)= d^{*}\Delta^{'}_{n,m+1}(e_i\ot e_j)$$
for every $n$ and $m$.
\end{proof}

The comultiplicative map $\displaystyle{\Delta_{n,\bullet}: \mathbb{Q}_{\bullet} \rightarrow \mathbb{Q}_{\bullet}\ott \mathbb{Q}_{\bullet}}$ that we want is a chain map lifting the identity map
$\mathbf{1}:\Lambda_n\rightarrow\Lambda_n\ott\Lambda_n\cong \Lambda_n$. Since both $\Delta^{'}_{n,\bullet}$ and $\Delta_{n,\bullet}$ are chain maps in the category of chain complexes of $\Lambda_n$, they are homotopic if and only if there is a chain homotopy $\hh^{\bullet}:\mathbb{Q}_{\bullet}\rightarrow \mathbb{P}_{\bullet+1}$ such that
\begin{equation}\label{homotopy-equation}
\Delta_{n,\bullet} - \Delta^{'}_{n,\bullet} = \hh^{\bullet-1}\partial^{\bullet} + d^{\bullet+1}\hh^{\bullet}
\end{equation}
Such a chain homotopy ensures that $\Delta^{'}_{n,\bullet}$ and $\Delta_{n,\bullet}$ induces the same map at the cohomology level. Moreover, there is some degree of freedom in how the chain homotopy can be defined once we have defined $\Delta^{'}_{n,\bullet}$.

\begin{theorem}\label{comultiplicative_map}
Let $\Lambda_n = \kk Q/I_n$, ($char(\kk)\neq 2$) be the family of quiver algebras of \eqref{family}. A comultiplicative map 
$\displaystyle{\Delta_{n,\bullet}: \mathbb{Q}_{\bullet} \rightarrow \mathbb{Q}_{\bullet}\ott \mathbb{Q}_{\bullet}},$ lifting the identity map $\mathbf{1}:\Lambda_n\rightarrow\Lambda_n\ott\Lambda_n\cong \Lambda_n$
on the projective bimodule resolution $\mathbb{Q}_{\bullet}$ of $\Lambda_n$  is given by
$$\Delta_{n,\bullet}  = \Delta^{'}_{n,\bullet} + \hh^{\bullet-1}\partial^{\bullet} + d^{\bullet+1}\hh^{\bullet}$$
where $\Delta^{'}_{n,\bullet}$ is the chain map lifting multiplication by 2 and $\hh^{\bullet}$ is a chain homotopy.
\end{theorem}
\begin{proof}
We have for every $m$, 
\begin{align*}
\Delta_{n,m}\partial^{m+1} &= \Delta^{'}_{n,m}\partial^{m+1} + \hh^{m-1}\partial^{m}\partial^{m+1} + d^{m+1}\hh^{m}\partial^{m+1} \\
&= \Delta^{'}_{n,m}\partial^{m+1} + d^{m+1}\hh^{m}\partial^{m+1}
\end{align*}
and 
\begin{align*}
d^{m+1}\Delta_{n,m+1} &= d^{m+1}\Delta^{'}_{n,m+1} + d^{m+1}\hh^{m}\partial^{m+1} + d^{m+1} d^{m+2}\hh^{m+1}\\
&= d^{m+1}\Delta^{'}_{n,m+1} + d^{m+1}\hh^{m}\partial^{m+1},
\end{align*}
so that $\Delta_{n,m}\partial^{m+1} - d^{m+1}\Delta_{n,m+1} =  \Delta^{'}_{n,m}\partial^{m+1} - d^{m+1}\Delta^{'}_{n,m+1} = 0$ by the result of Lemma \ref{deltap-chain_map}.
\end{proof}

\begin{remark}\label{acyclic-condition}
Notice the difference of the maps $\mathbf{1}-\pi=\mathbf{1}$, therefore, $\Delta_{n,\bullet} - \Delta^{'}_{n,\bullet}$ is a chain map lifting the identity map and $\hh^{\bullet}$ is a chain contraction of the identity map. Therefore the resolution $\mathbb{Q}_{\bullet}$ is acyclic\cite{witherspoon2019hochschild}. Thus, we realize the results of \cite[Theorem 3.3]{furuya2012projective}
\end{remark}

\begin{Proposition}
    The projective bimodule resolution $\mathbb{Q}_{\bullet}$ is acyclic.
\end{Proposition}
\begin{proof}
    Follows from Remark \ref{acyclic-condition}
\end{proof}

\subsection{Comultiplicative structure:} We recall a basis for $\HHom_{\Lambda_n^e}(\mathbb{Q}_{\bullet},\Lambda_n)$ according to \cite[Lemma 4.1]{furuya2012projective}. We start with a basis for $\HHom_{\Lambda_n^e}(\mathbb{Q}_{0},\Lambda_0)$. This is the case where $m=0$ and $n=0$. If $\phi(e_i\ot e_j)=w$, then $\mfo(w)=e_i$ and $\mft(w)=e_j$ so that as $e_i\phi(e_i\ot e_j)e_j=e_iwe_j=w$ . We know that since the set $\{e_0\ot e_0, e_1\ot e_1, f_1\ot f_1, e_2\ot e_2\}$ form a basis for $\mathbb{Q}_0 = \bigoplus_{g\in \mathcal{G}^0}\Lambda_0\mfo(g)\ot\mft(g)\Lambda_0$ as a bimodule for all $n$, the images of $e_i\ot e_j$ under any bimodule homomorphism $\phi:\mathbb{Q}_{0}\rightarrow\Lambda_0$ are elements whose origin vertex is $e_i$ and whose terminal vertex is $e_j$. This means we have the set \{ $\alpha^0_0: e_0\ot e_0\mapsto e_0$, $\alpha^0_1: e_1\ot e_1\mapsto e_1$, $\alpha^0_2: e_2\ot e_2\mapsto e_2$, and $\beta: f_1\ot f_1\mapsto f_1$, 0\;\text{otherwise}\} of bimodule homomorphisms as a basis for $\HHom_{\Lambda_0^e}(\mathbb{Q}_{0},\Lambda_0)$: we write 

$$\alpha_i^0: 
     \begin{cases}
     e_0\ot e_0 \mapsto e_0 & i=0\\
     e_1\ot e_1 \mapsto e_1 & i=1\\
     e_2\ot e_2 \mapsto e_2 & i=2\\
     e_i\ot e_j \mapsto 0 & i\neq j
     \end{cases}, \beta:\begin{cases}
     f_1\ot f_1 \mapsto f_1\\
     e_i\ot e_j \mapsto 0 & i\neq j
     \end{cases}$$

and $0$ otherwise. For general $n$, a generating set for $\HHom_{\Lambda_n^e}(\mathbb{Q}_{0},\Lambda_n)$ is the set $\{\alpha_s^t, \beta |\; s=0,1,2,\; t=0,1,2,\ldots,n\}$ given on $e_l\ot e_l$, $l=0,1,2$ and $f_1\ot f_1$ by
$$\alpha_s^t: 
     \begin{cases}
     e_l\ot e_l \mapsto \begin{cases} (a_sa_{s+1}a_{s+2})^t& s=l\\
      0& s\neq l
     \end{cases}\\
     f_1\ot f_l \mapsto 0 & 
     \end{cases},$$
     
$$\beta: \begin{cases}
    f_1\ot f_1 \mapsto f_1\\
    e_l\ot e_l \mapsto 0 &l=0,1,2\\
\end{cases}$$     

Also, the following homomorphisms form a basis for $\HHom_{\Lambda_{n}^e} (\mathbb{Q}_{3m},\Lambda_n)$ with $m>0$;  $\{\phi_i^t, \psi |\; i=0,1,2,\; t=0,1,2,\ldots,n\}$ given by 
$$\phi_i^t: \begin{cases}
    e_l\ot e_l \mapsto \begin{cases}
    (a_ia_{i+1}a_{i+2})^t & i=l\\
     0 &i\neq l\\
\end{cases}\\
    f_1\ot f_1 \mapsto 0\\
    f_1\ot e_1 \mapsto 0\\
    e_l\ot f_l \mapsto 0 &
\end{cases} \qquad
\psi:\begin{cases}
    f_1\ot f_1 \mapsto f_1\\
    f_1\ot e_1 \mapsto 0\\
    e_0\ot f_1 \mapsto 0\\
    e_l\ot e_l \mapsto 0 & \text{for all}\hspace{0.1cm} l
\end{cases}$$

A basis for $\HHom_{\Lambda_{n}^e} (\mathbb{Q}_{3m+1},\Lambda_n)$ with $m\geq0$ is the set $\{\mu_i^t, \nu_0, \nu_1 |\; i=0,1,2,\; t=0,1,2,\ldots,n\}$ defined by 
$$\mu_i^t:\begin{cases}
    e_l\ot e_{l+1} \mapsto  (a_ia_{i+1}a_{i+2})^ta_i & i=l\\
    e_l\ot e_{l+1} \mapsto 0 & i\neq l\\
    e_0\ot f_1 \mapsto 0\\
    f_1\ot e_2 \mapsto 0\\
    f_1\ot f_1 \mapsto 0
\end{cases} \qquad \nu_0:\begin{cases}
    e_0\ot f_1\mapsto b_0\\
    f_1\ot e_2\mapsto 0\\
    e_s\ot e_l\mapsto 0 & \text{for all}\hspace{0.1cm} s, l\\
    f_1\ot f_1\mapsto 0
\end{cases}$$

$$\nu_1: \begin{cases}
    e_0\ot f_1\mapsto 0\\
    f_1\ot e_2 \mapsto b_1\\
    e_s\ot e_l \mapsto 0 & \text{for all}\hspace{0.1cm} s, l\\
    f_1\ot f_1 \mapsto 0.
\end{cases}$$

A basis for $\HHom_{\Lambda_{n}^e} (\mathbb{Q}_{3m+2},\Lambda_n)$ with $m\geq0$ is the set $\{\theta_i^t, \eta |\; i=0,1,2,\; t=0,1,2,\ldots,n-1\}$ with $\theta_i^t: e_l\ot e_{l+2}\mapsto (a_ia_{i+1}a_{i+2})^ta_ia_{i+1}$ if $i=l$ and $0$ if $i\neq l$, $\theta_i^t: f_1\ot e_0\mapsto 0, e_2\ot f_1\mapsto 0,$ and $\eta: e_0\ot e_{2}\mapsto (a_0a_1a_2)^na_0a_1, e_1\ot e_0\mapsto 0, e_2\ot e_1\mapsto 0, f_1\ot e_0\mapsto 0, e_2\ot f_1\mapsto 0$.

$$\theta_i^t:\begin{cases}
    e_l\ot e_{l+2}\mapsto (a_ia_{i+1}a_{i+2})^ta_ia_{i+1} & i=l\\
    e_l\ot e_{l+2} \mapsto 0 & i\neq l\\
    f_1\ot e_0 \mapsto 0\\
    e_2\ot f_1\mapsto 0\\
    f_1\ot f_1\mapsto 0
\end{cases} \qquad \eta:\begin{cases}
    e_0\ot e_2\mapsto (a_0a_1a_2)^na_0a_1\\
    e_s\ot e_l\mapsto 0 & s\neq l\\
    f_1\ot e_0\mapsto 0\\
    e_2\ot f_1\mapsto 0
\end{cases}$$

The following lemma presents a summary of the above result.
\begin{lemma}\cite[]{furuya2012projective}\label{Hochschild_dim}
For $n\geq 0$ and $m\geq 0$, the generators for $\HHom_{\Lambda_n^e}(\mathbb{Q}_{*},\Lambda_n)$ are given as follows:
\begin{enumerate}
\item A $\kk$-basis for $\HHom_{\Lambda_n^e}(\mathbb{Q}_{0},\Lambda_n)$ is the set $$\{\alpha_s^t, \beta |\; s=0,1,2,\; t=0,1,2,\ldots,n\}$$
\item A $\kk$-basis for $\HHom_{\Lambda_n^e}(\mathbb{Q}_{3m},\Lambda_n)$ is the set $$\{\phi_i^t, \psi |\; i=0,1,2,\; t=0,1,2,\ldots,n\}$$
\item A $\kk$-basis for $\HHom_{\Lambda_n^e}(\mathbb{Q}_{3m+1},\Lambda_n)$ is the set $$\{\mu_i^t, \nu_s |\; i=0,1,2,\; t=0,1,2,\ldots,n,\; s=0,1\}$$
\item A $\kk$-basis for $\HHom_{\Lambda_n^e}(\mathbb{Q}_{3m+2},\Lambda_n)$ is the set $$\{\theta_i^t, \eta |\; i=0,1,2,\; t=0,1,2,\ldots,n-1\}$$
\end{enumerate}
\end{lemma}

The multiplicative structure on Hochschild cohomology $\HH^*(\Lambda_n)$ for this family involves using the cup product of Definition \eqref{cup-definition}. This involves the composition of the following maps at the chain level in which for any $f \in \HHom_{\Lambda_n^e}(\mathbb{Q}_{a},\Lambda_n)$ and $g \in \HHom_{\Lambda_n^e}(\mathbb{Q}_{b},\Lambda_n)$, $f\smallsmile g\in \HHom_{\Lambda_n^e}(\mathbb{Q}_{a+b},\Lambda_n)$ is given by the following composition of maps;
$$\mathbb{Q}{\bullet}\xrightarrow{\Delta} \mathbb{Q}_{\bullet}\ot_{\Lambda_n} \mathbb{Q}_{\bullet} \xrightarrow{f\ot g} \Lambda_n \ot_{\Lambda_n}\Lambda_n \xrightarrow{} \Lambda_n,$$
that is $f\smallsmile g:= (f\ot g)\Delta_{n,a+b}$ where the last map is multiplication. We use the fact that Hochschild cohomology is graded commutative to evaluate $g\smallsmile f= (-1)^{ab}f\smallsmile g$. Furthermore, define the relationship between the cup product and the star product in degree $a+b=m$ as
\begin{align*}\label{cup-product-delta}
    f\smallsmile g &= (f\ot g)\Delta_{n,m}\\
    =&(f\ot g)(\Delta'_{n,\bullet} + \hh^{m-1}\partial^m+d^{m+1}\hh^m)\\
    =& f\star g +(f\ot g)(\hh^{m-1}\partial^m+d^{m+1}\hh^m)
\end{align*}
We obtain a general result for the star product $f\star g$ in Theorem \ref{comultiplicative_structure} and present the cup product on Hochschild cohomology for $\Lambda_0$. We will need the following definition in order to present $\HH^*(\Lambda_0).$
\begin{definition}\label{contracting-homo-exam}
Let $\Lambda_n = \kk Q/I_n$ be the family of quiver algebras of \eqref{family} and let $w_i$ be any idempotent in the set $\{e_0,e_1,f_1,e_2\}$ with the index $i$ given mode 3. Then a chain homotopy $\hh^{\bullet}$ given as $\hh^*:\Lambda_n\rightarrow \mathbb{P}_{1}$, $\hh^0:\mathbb{Q}_{3m}\rightarrow \mathbb{P}_{3m+1}$, $\hh^1:\mathbb{Q}_{3m+1}\rightarrow \mathbb{P}_{3m+2}$ and $\hh^2:\mathbb{Q}_{3m+2}\rightarrow \mathbb{P}_{3m}$ can be defined in the following ways;
\begin{align*}
\hh^*(w_i) &=\begin{cases} -(w_i\ot w_i) \bar{\ot} (w_i\ot w_{i})a_i, & \mfo(a_i)=w_i\\
(w_i\ot w_i) \bar{\ot} (w_i\ot w_{i})b_i, & \mfo(b_i)=w_i
\end{cases}\\
\hh^0(w_i\ot w_i) &= (w_i\ot w_i) \bar{\ot} (w_i\ot w_{i+1}) \\
\hh^1(w_i\ot w_{i+1}) &= (w_i\ot w_i) \bar{\ot} (w_i\ot w_{i+2}) \\
\hh^2(w_i\ot w_{i+2}) &= (w_i\ot w_i) \bar{\ot} (w_i\ot w_{i+3}).
\end{align*}
\end{definition}

The following diagram demonstrates $\hh^{\bullet}.$
\[
\begin{tikzcd}[row sep=large, column sep=huge]
\cdots \arrow[r, "\partial^2"] 
  & \mathbb{Q}_1 \arrow[r, "\partial^1"] \arrow[d, "\Delta_{n,1}", left=50] \arrow[d, "\Delta^{'}_{n,1}"', bend right=20] \arrow[dl, Rightarrow, "\hh^1" description]
  & \mathbb{Q}_{0} \arrow[r, "\partial^0"] \arrow[d, "\Delta_{n,0}", left=50] \arrow[d, "\Delta^{'}_{n,0}"', bend right=20] \arrow[dl, Rightarrow, "\hh^0" description]
  & \Lambda_n \arrow[d, "\mathbf{1}", left=50] \arrow[d, "\pi"', bend right=20] \arrow[dl, Rightarrow, "\hh^*" description]\\
\cdots \arrow[r, "d^2"]
  & \mathbb{P}_1 \arrow[r, "d^1"]
  & \mathbb{P}_0 \arrow[r, "d^0"]
  & \Lambda_n\ott\Lambda_n
\end{tikzcd}
\]
\begin{example}\label{example-of-chain-map}
The following are some examples of images of chain maps $\Delta_{n,\bullet}$ and $\Delta^{'}_{n,\bullet}$ ($n=0$) on some basis elements. We make use of homotopy chain maps of Definition \ref{contracting-homo-exam}.
\begin{align*}
&\Delta_{n,0}(e_0\ot e_0) = (\Delta^{'}_{n,0}+d^1\hh^0 + \hh^*\partial^0)(e_0\ot e_0)\\
&= \Delta^{'}_{n,0}(e_0\ot e_0) + d^1((e_0\ot e_0)\ott (e_0\ot e_1)) + \hh^*(e_0) \\
&= \Delta^{'}_{n,0}(e_0\ot e_0) + (e_0\ot e_0)\ott (e_0\ot e_0a_0 - a_1e_1\ot e_1)\\
&- (e_0\ot e_0)\ott (e_0\ot e_0)a_0 \\
&= \Delta^{'}_{n,0}(e_0\ot e_0) + (e_0\ot e_0)\ott (e_0\ot e_0)a_0 - a_1(e_0\ot e_0)\ott (e_0\ot e_0) \\
&- (e_0\ot e_0)\ott (e_0\ot e_0)a_0 \\
&= \Delta^{'}_{n,0}(e_0\ot e_0).
\end{align*}

Similar computations show that
\begin{align*}
\Delta_{n,0}(e_1\ot e_1)&=\Delta'_{n,0}(e_1\ot e_1)\\
\Delta_{n,0}(e_2\ot e_2)&=\Delta'_{n,0}(e_2\ot e_2)\\
\Delta_{n,0}(f_1\ot f_1)&=\Delta'_{n,0}(f_1\ot f_1)+(f_1\ot f_1)\ott (f_1\ot e_1) a_1
\end{align*}
and at degree one, we have
\begin{align*}
&\Delta_{n,1}(e_1\ot e_2)=(\Delta_{n,1}'+d^2\hh^1+\hh^0\partial^1)(e_1\ot e_2)\\
&=\Delta_{n,1}'(e_1\ot e_2)+d^2((e_1\ot e_1)\ott (e_1\ot e_0))+\hh^0(e_1\ot e_1a_1- a_1e_2\ot e_2)\\
&=\Delta'_{n,1}(e_1\ot e_2)+(\partial\ot 1 +(-1)^\ast 1\ot \partial) (e_1\ot e_1)\ott (e_1\ot e_0) \\
&+\hh^0(e_1\ot e_1)a_1-a_1\hh^0(e_2\ot e_2)\\
&=\Delta_{n,1}'(e_1\ot e_2)+0+ (-1)^0(e_1\ot e_1)\ott \partial (e_1\ot e_0)+(e_1\ot e_1\ott e_1\ot e_2)a_1\\
&-a_1(e_2\ot e_2\ott e_2\ot e_0)\\
&= \Delta_{n,1}'(e_1\ot e_2)+(e_1\ot e_1)\ott (e_1\ot e_2)a_2 - a_1(e_1\ot e_1)\ott (e_2\ot e_0) \\
&+ (e_1\ot e_1\ott e_1\ot e_2)a_1-a_1(e_2\ot e_2\ott e_2\ot e_0)\\
&= \Delta_{n,1}'(e_1\ot e_2)+(e_1\ot e_1)\ott (e_1\ot e_2)a_2 -a_1(e_2\ot e_2\ott e_2\ot e_0)
\end{align*}

Similar computations show that
\begin{align*}
\Delta_{n,1}(e_0\ot e_1)&=\Delta'_{n,1}(e_0\ot e_1) + (e_0\ot e_0)\ott (e_0\ot e_1) a_1 \\
&- (e_0\ot e_0)\ott (e_0\ot f_1) b_1 
 - a_0(e_1\ot e_1)\ott (e_1\ot e_2).\\
\Delta_{n,1}(e_0\ot f_1)&=\Delta'_{n,1}(e_0\ot f_1) + (e_0\ot e_0)\ott (e_0\ot f_1)b_1 + b_0(f_1\ot f_1\ott f_1\ot e_2)
\end{align*}
\begin{align*}
\Delta_{n,1}(e_2\ot e_0)&=\Delta'_{n,1}(e_2\ot e_0)+ (e_2\ot e_2)\ott (e_2\ot e_0)a_0 - a_2(e_0\ot e_0\ott e_0\ot e_1) \\
\Delta_{n,1}(f_1\ot e_2)&=\Delta'_{n,1}(f_1\ot e_2)+ (f_1\ot f_1)\ott (f_1\ot e_0)b_0 - b_1(e_2\ot e_2\ott e_2\ot e_0).
\end{align*}
\end{example}

\begin{Proposition}\label{cohomology-L-0}
    $$\HH^*(\Lambda_0) = \kk\langle x, y, z\rangle /(x^2-2x,xy-y,yx-y,zx-z, xz-z,z^2-x, y^2, yz, zy)$$ 
    where the degrees of $x,y,z$ are $0,1$ and $6$ respectively.
\end{Proposition}
\begin{proof}
We have $\Lambda_0 = \kk\mathcal{Q}/\langle  a_0a_1-b_0b_1, a_1a_2,a_2a_0, b_1b_2, b_2b_0 \rangle$. From Lemma \ref{Hochschild_dim}, the set of maps $\{\alpha_0^0,\alpha_1^0, \alpha_2^0, \beta\}$ forms a generating set for $\HHom(\mathbb{Q}_0,\Lambda_0)$. Moreover for $t\in\mathbb{Q}_1$, $t$ is a linear combination of basis elements. Write $t=e_0\ot e_1+e_0\ot f_1+ e_1\ot e_2 + f_1\ot e_2 + e_2\ot e_0$, calculations shows that
\begin{align*}
&(\alpha_0^0+\alpha_1^0+\alpha_2^0+\beta)(\partial^1(t)) = (\alpha_0^0+\alpha_1^0+\alpha_2^0+\beta)(e_0\ot e_0a_0 - a_0 e_1\ot e_1 \\
&+ e_0\ot e_0b_0 - b_0 f_1\ot f_1 + e_1\ot e_1a_1 - a_1 e_2\ot e_2 + f_1\ot f_1b_1 - b_1 e_2\ot e_2 \\
&+ e_2\ot e_2a_2 - a_2 e_0\ot e_0) = e_0a_0 - a_0e_1 + f_0b_0-b_0f_1+e_1a_1-a_1e_2\\
&+f_1b_1-b_1f_2+e_2a_2-a_2e_0 = 0,
\end{align*}
showing that $(\alpha_0^0+\alpha_1^0+\alpha_2^0+\beta):=x$ generates $\HH^0(\Lambda_0)$. Further calculations shows that $(\mu_0^0+\nu_0):=y$ generates $\HH^1(\Lambda_0)$, and
$(\phi_0^0+\phi_1^0+\phi_2^0-\psi):=z$ generates $\HH^6(\Lambda_0)$ when $char(\kk)\neq 2$ \cite[Proposition 4.9]{furuya2012projective}. The second, third, fourth and fifth Hochschild cohomology groups are zero i.e. $\HH^2(\Lambda_0) = \HH^3(\Lambda_0) = \HH^4(\Lambda_0) = \HH^5(\Lambda_0) = 0$. Using $\Delta_{n,m}$ of Example \ref{example-of-chain-map}, we obtain the following table of multiplication at the chain level
\begin{tabular}{|c|c|c|c|} \hline
& $x\in\HHom(\mathbb{Q}_{0},\Lambda_0)$ & $y\in \HHom(\mathbb{Q}_{1},\Lambda_0)$ & $z\in\HHom(\mathbb{Q}_{6},\Lambda_0)$ \\\hline
$x\in\HHom(\mathbb{Q}_{0},\Lambda_0)$ & 2x & y & z \\\hline
$y\in\HHom(\mathbb{Q}_{0},\Lambda_0)$ & y & 0 & 0 \\\hline
$z\in\HHom(\mathbb{Q}_{0},\Lambda_0)$ & z & 0 & x \\\hline
\end{tabular}
So we have $\HH^*(\Lambda_0) = \kk\langle x, y, z\rangle /(x^2-2x,xy-y,yx-y,zx-z, xz-z,z^2-x, y^2, yz, zy)$ where the degrees of $x$ is $0$, degree of $y$ is 1 and the degree of $z$ is 6.
\end{proof}
We present the following general star product for chain maps $f$ and $g$, $(f\star g)(e_i\ot e_j) = (f\ot g)\Delta_{n,*}')(e_i\ot e_j)$ which can be used to define the cup product on Hochschild cohomology. That is
\begin{equation}
f\smallsmile g = f\star g + (f\ot g)(d\hh + \hh\partial).
\end{equation}

\begin{theorem}\label{comultiplicative_structure}
Let $\Lambda_n = \kk Q/I_n$ be the family of quiver algebras of \eqref{family} with projective bimodule resolution $\mathbb{Q}_{\bullet}$. Suppose that a $\kk$-basis for $\HHom_{\Lambda_n^e}(\mathbb{Q}_{0},\Lambda_n)$ is
$\{\alpha_s^t, \beta |\; s=0,1,2,\; t=0,1,2,\ldots,n\}$,
a basis for $\HHom_{\Lambda_n^e}(\mathbb{Q}_{3m},\Lambda_n)$ is
$\{\phi_i^t, \psi |\; i=0,1,2,\; t=0,1,2,\ldots,n\}$,
a basis for $\HHom_{\Lambda_n^e}(\mathbb{Q}_{3m+1},\Lambda_n)$ is
$\{\mu_i^t, \nu_s |\; i=0,1,2,\; t=0,1,2,\ldots,n,\; s=0,1\}$
and a $\kk$-basis for\\ $\HHom_{\Lambda_n^e}(\mathbb{Q}_{3m+2},\Lambda_n)$ is 
$\{\theta_i^t, \eta |\; i=0,1,2,\; t=0,1,2,\ldots,n-1\}$. 
The following gives the star product $(f\star g):= (f\ot g)\Delta_{n,*}'$ at the chain level for any $f,g\in\HHom_{\Lambda_n^e}(\mathbb{Q}_{a},\Lambda_n)$, $a=3m,3m+1,3m+2.$
\end{theorem}
$$
f\star\alpha_{s'}^{t'} = \begin{cases} 
\alpha_s^{t+t'}, & \quad f=\alpha_s^t, \quad s=s'\\
\phi_s^{t+t'}, & \quad f=\phi_i^t, \quad i=s'\\
\mu_s^{t+t'}, & \quad f=\mu_i^t, \quad i=s'\\
\nu_1, & \quad f=\nu_1,\\
\theta_s^{t+t'}, & \quad f=\theta_i^t, \quad i=s'\\
\eta, & \quad f=\eta,\\
0 & \quad \text{otherwise}
\end{cases},
f\star\phi_{i'}^{t'} = \begin{cases} 
\phi_i^{t+t'}, & \quad f=\phi_i^t, \quad i=i'\\
\mu_i^{t+t'}, & \quad f=\mu_i^t, \quad i=i'\\
\nu_1, & \quad f=\nu_1,\\
\theta_i^{t+t'}, & \quad f=\theta_i^t, \quad i=i'\\
\eta, & \quad f=\eta,\\
0 & \quad \text{otherwise}
\end{cases}
$$
$$
f\star\beta = \begin{cases} 
\beta, & \quad f=\beta\\
\psi, & \quad f=\psi\\
\nu_0, & \quad f=\nu_0\\
0 & \quad \text{otherwise}
\end{cases}
f\star\psi = \begin{cases} 
\psi, & \quad f=\psi\\
\nu_0, & \quad f=\nu_0\\
0 & \quad \text{otherwise}
\end{cases}.
$$
\begin{proof}
We will present the star product computations for $f\star \alpha_{s'}^{t'}$ and $f\star \beta$ for each $f\in\HHom_{\Lambda_n^e}(\mathbb{Q}_{*},\Lambda_n)$ using $(f\ot g)\Delta'_{n,*}$. The other computations are done in similar ways. We have
\begin{align*}
&(\alpha_s^t\star \alpha_{s'}^{t'})(e_l\ot e_l) = (\alpha_s^t\ot\alpha_{s'}^{t'})\Delta'(e_l\ot e_l)\\
&= (\alpha_s^t\ot\alpha_{s'}^{t'})[(e_l\ot e_l)\ott (e_l\ot e_l) + \sum_{w\neq e_l} (w\ot w)\ott (e_l\ot e_l) + (e_l\ot e_l)\ott (w\ot w) ]\\
&= (\alpha_s^t(e_l\ot e_l)\ot\alpha_{s'}^{t'}(e_l\ot e_l)) = (a_sa_{s+1}a_{s+2})^t(a_sa_{s+1}a_{s+2})^{t'} = (a_sa_{s+1}a_{s+2})^{t+t'}\\
&= \alpha_s^{t+t'}(e_l\ot e_l),
\end{align*}
provided $t+t'\leq n$. We note that $(\alpha_s^t\star \alpha_{s'}^{t'})(e_l\ot e_l) =0$ if $s\neq s'$ or $t+u>n$. so $(\alpha_s^t\star \alpha_{s'}^{t'})= \alpha_{s}^{t+t'}, (s=s')$. In similar fashion,
\begin{align*}
&(\alpha_s^t\star \beta)(e_l\ot e_l) = (\alpha_s^t\ot\beta)\Delta'(e_l\ot e_l)\\
&= (\alpha_s^t\ot\beta)[(e_l\ot e_l)\ott (e_l\ot e_l) + \sum_{w\neq e_l} (w\ot w)\ott (e_l\ot e_l) + (e_l\ot e_l)\ott (w\ot w)]\\
&= (\alpha_s^t(e_l\ot e_l)\ot\beta(f_1\ot f_1)) = (a_sa_{s+1}a_{s+2})^tf_1 = 0,
\end{align*}
so $(\alpha_s^t\star \beta)= 0$, since the terminal vertex of $a_sa_{s+1}a_{s+2}$ is either vertex $e_1$ or $e_2$. For this case, 
\begin{align*}
&(\beta\star \beta)(f_1\ot f_1) = (\beta\ot\beta)(\Delta')(f_1\ot f_1)\\
&= (\beta\ot\beta)[(f_1\ot f_1)\ott (f_1\ot f_1) + \sum_{w\neq f_1} (w\ot w)\ott (f_1\ot f_1) + (f_1\ot f_1)\ott (w\ot w)]\\
&= (\beta(f_1\ot f_1)\ot\beta(f_1\ot f_1)) = f_1f_1 = f_1 = \beta(f_1\ot f_1).
\end{align*}
\begin{align*}
    &(\mu_i^t \star \alpha_{s'}^{t'})(e_l\ot e_{l+1})=(\mu_i^t\ot\alpha_{s}^{t'})\Delta'(e_{l}\ot e_{l+1})\\
    &=(\mu_i^t\ot \alpha_{s}^{t'})[(e_l\ot e_l)\ott (e_l\ot e_{l+1}) + (e_{l}\ot e_{l+1})\ott (e_{l+1}\ot e_{l+1}) + \\&(e_l\ot e_{l+1})\ott (e_l\ot e_l)+ (e_{l+1}\ot e_{l+1})\ott (e_l\ot e_{l+1}) \\
    & + \sum_{w\neq e_l, e_{l+1}}(w\ot w)\ott (e_l\ot e_l) + (e_l\ot e_l)\ott (w\ot w)]\\
    &=(\mu_i^t(e_l\ot e_{l+1})\ott \alpha_{s'}^{t'}(e_{l+1}\ot e_{l+1})
    +\mu_i^t(e_l\ot e_{l+1})\ott \alpha_{s'}^{t'}(e_l\ot e_l))
\end{align*}
\begin{itemize}
\item Case 1: If $i=l$ and $s'=l+1$, we have,
$((a_la_{l+1}a_{l+2})^ta_l(a_{l+1}a_{l+2}a_{l})^{t'}+0)=(a_la_{l+1}a_{l+2})^{t+t'}a_l=\mu_i^{t+t'}(e_l\ot e_{l+1})$, so $(\mu_i^t\smallsmile \alpha_{s'}^{t'})= \mu_{i}^{t+t'}, (i=l)$.
\item  Case 2: If $i=l=s'$, we have $(a_la_{l+1}a_{l+2})^ta_l(a_la_{l+1}a_{l+2})^{t'}=\mu_i^{t+t'}(e_l\ot e_{l+1}),$ so that $(\mu_i^t \star \alpha_{s'}^{t'})=\mu_i^{t+t'}$.
\end{itemize}
\begin{align*}
    &(\mu_i^t\star \beta)(e_l\ot e_{l+1})=\pi(\mu_i^t\ot \beta)\Delta'(e_l\ot e_{l+1})\\
    &=(\mu_i^t\ot \beta)[(e_l\ot e_l)\ott (e_l\ot e_{l+1}) + (e_{l}\ot e_{l+1})\ott (e_{l+1}\ot e_{l+1}) + \\&(e_l\ot e_{l+1})\ott (e_l\ot e_l)+ (e_{l+1}\ot e_{l+1})\ott (e_l\ot e_{l+1}) \\
    &+ \sum_{w\neq e_l, e_{l+1}}(w\ot w)\ott (e_l\ot e_l) + (e_l\ot e_l)\ott (w\ot w)]\\
    &=(\mu_i^t(e_l\ot e_{l+1})\ott \beta(e_{l+1}\ot e_{l+1})+\mu_i^t(e_l\ot e_{l+1})\ott \beta(e_l\ot e_l))\\
    &=0,\text{so we have} \;(\mu_i^t\star \beta)= 0.
\end{align*}
\begin{align*}
    &(\mu_i^{t} \star \mu_{i'}^{t'})(e_l\ot e_{l+1})=\pi(\mu_i^t\ot \mu_{i'}^{t'})\Delta'(e_{l}\ot e_{l+1})\\
    &=(\mu_i^t\ot \mu_{i'}^{t'})[(e_l\ot e_l)\ott (e_l\ot e_{l+1}) + (e_{l}\ot e_{l+1})\ott (e_{l+1}\ot e_{l+1}) + \\&(e_l\ot e_{l+1})\ott (e_l\ot e_l)+ (e_{l+1}\ot e_{l+1})\ott (e_l\ot e_{l+1})\\
    &+ \sum_{w\neq e_l, e_{l+1}}(w\ot w)\ott (e_l\ot e_l) + (e_l\ot e_l)\ott (w\ot w)]\\
    &=0
\end{align*}
\begin{align*}
&(\theta_i^t\star \alpha_{s'}^{t'})(e_l\ot e_{l+2})=\pi(\theta_i^t \ot \alpha_{s'}^{t'})\Delta'(e_l\ot e_{l+2})\\
&=(\theta_i^t\ot \alpha_{s'}^{t'})[(e_l\ot e_l)\ott (e_l\ot e_{l+2})+(e_l\ot e_{l+1})\ott (e_{l+1}\ot e_{l+2})\\
&+(e_l\ot e_{l+2})\ott (e_{l+2}\ot e_{l+2})]+ (e_{l+2}\ot e_{l+2})\ott (e_l\ot e_{l+2}) \\
&+ (e_{l+1}\ot e_{l+2})\ott (e_l\ot e_{l+1})+ (e_l\ot e_{l+2})\ott (e_{l}\ot e_{l})\\ &+\sum_{w\neq e_l,e_{l+1}, e_{l+2}} (w\ot w)\ott (e_l\ot e_l) + (e_l\ot e_l) \ott (w\ot w)]\\
&=(\theta_i^t(e_l\ot e_{l+2})\ott \alpha_{s'}^{t'}(e_l\ot e_l))\\
&=(a_la_{l+1}a_{l+2})^ta_la_{l+1} (a_la_{l+1}a_{l+2})^{t'}=(a_la_{l+1}a_{l+2})^{t+t'}a_{l}a_{l+1}=\theta_i^{t+t'}(e_l\ot e_{l+2}),\\
&\text{so we have}\;(\theta_i^t\star \alpha_{s'}^{t'})= \theta_{i}^{t+t'}, (s'=i+2).
\end{align*}
\begin{align*}
    &(\theta_i^t\star \beta)(e_l\ot e_{l+2})=\pi(\theta_i^t\ot \beta)\Delta'(e_l\ot e_{l+2})\\
    &=(\theta_i^t\ot \beta)[(e_l\ot e_l)\ott (e_l\ot e_{l+2})+(e_l\ot e_{l+1})\ott (e_{l+1}\ot e_{l+2})\\
    & +(e_l\ot e_{l+2})\ott (e_{l+2}\ot e_{l+2})]+\\
&(e_{l+2}\ot e_{l+2})\ott (e_l\ot e_{l+2})+ (e_{l+1}\ot e_{l+2})\ott (e_l\ot e_{l+1})+ (e_l\ot e_{l+2})\ott (e_{l}\ot e_{l})\\ &+\sum_{w\neq e_l,e_{l+1}, e_{l+2}} w\ot w\ott e_l\ot e_l + e_l\ot e_l \ott w\ot w]=0, (\theta_i^t\smallsmile \beta)=0.
\end{align*}
\begin{align*}
    &(\nu_0\star \alpha_{s'}^{t'})(e_0\ot f_1)=(\nu_0\ot \alpha_{s'}^{t'})\Delta'(e_0\ot f_1)\\
    &=(\nu_0\ot \alpha_{s'}^{t'})[(e_0\ot e_0)\ott (e_0\ot f_1) + (e_0\ot f_1)\ott (f_1\ot f_1) + \\
    &(e_0\ot f_1)\ott (e_0\ot e_0) + (f_1\ot f_1)\ott (e_0\ot f_1)\\
    & + \sum_{w\neq e_0, f_1} (w\ot w) \ott (e_l\ot e_l) + (e_l\ot e_l) \ott (w\ot w)]\\
    &=((\nu_0(e_0\ot f_1)\ott \alpha_{s'}^{t'}(e_0\ot e_0))=b_0(a_0a_{1}a_{2})^t=0, \text{so we have  }\;(\nu_0\star \alpha_{s'}^{t'}) = 0.
\end{align*}
\begin{align*}
    &(\nu_0 \star \beta)(e_0\ot f_1)=(\nu_0\ot \beta)\Delta'(e_0\ot f_1)\\
    &=(\nu_0\ot \beta)[(e_0\ot e_0)\ott (e_0\ot f_1) + (e_0\ot f_1)\ott (f_1\ot f_1) + \\
    &(e_0\ot f_1)\ott (e_0\ot e_0) + (f_1\ot f_1)\ott (e_0\ot f_1) \\
    &+ \sum_{w\neq e_0, f_1} (w\ot w) \ott (e_l\ot e_l) + (e_l\ot e_l) \ott (w\ot w)]\\
    &=(\nu_0(e_0\ot f_1)\ott \beta (f_1\ot f_1)=b_0f_1 =b_0=\nu_0(e_0\ot f_1)
\end{align*}
\begin{align*}
    &(\phi_i^t \star \alpha_{s'}^{t'})(e_l\ot e_l)=\pi(\phi_i^t \ot \alpha_{s'}^{t'})\Delta'(e_l\ot e_l)\\
    &=(\phi_i^t\ot \alpha_{s'}^{t'})[e_l\ot e_l\ott e_l\ot e_l + \sum_{w\neq e_l} w\ot w\ott e_l\ot e_l + e_l\ot e_l\ott w\ot w]\\
    &=(\phi_i^t(e_l\ot e_l)\ot \alpha_{s'}^{t'}(e_l\ot e_l))\\
    &=((a_ia_{i+1}a_{i+2})^t\ot (a_{s'}a_{s'+1}a_{s'+2})^{t'})=(a_ia_{i+1}a_{i+2})^t(a_{s'}a_{s'+1}a_{s'+2})^{t'}\\
    &=(a_ia_{i+1}a_{i+2})^{t+t'}=\phi_i^{t+t'}(e_l\ot e_l), \text{ so } (\phi_i^t \star \alpha_{s'}^{t'})=\phi_i^{t+t'}.
\end{align*}
The second last equality follows from the fact that $i$ and $s'$ needs to be equal for the cup product $(\phi_i^t\star \alpha_{s'}^{t'})$ to be non-zero and the last equality holds as $\phi_i^t\star \alpha_{s'}^{t'}\in \mathbb{Q}_{3m}$.
\begin{align*}
    & (\phi_i^t\star \beta)(e_l\ot e_l)=(\phi_i^t\ot \beta)\Delta'(e_l\ot e_l)\\
    =& (\phi_i^t\ot \beta)[e_l\ot e_l\ott e_l\ot e_l + \sum_{w\neq e_l} w\ot w\ott e_l\ot e_l + e_l\ot e_l\ott w\ot w]\\
    &=(\phi_i^t(e_l\ot e_l)\ot \beta (f_1\ot f_1)) =(a_ia_{i+1}a_{i+2})^tf_1=0.
\end{align*}
\begin{align*}
    &(\psi \star \alpha_{s'}^{t'})(e_l\ot e_l)=(\psi \ot \alpha_{s'}^{t'})\Delta'(e_l\ot e_l))\\
    &=(\psi \ot \alpha_{s'}^{t'})[e_l\ot e_l\ott e_l\ot e_l + \sum_{w\neq e_l} w\ot w\ott e_l\ot e_l + e_l\ot e_l\ott w\ot w]\\
    &=(\psi(f_1\ot f_1)\ot \alpha_{s'}^{t'}(e_l\ot e_l))=f_1(a_{s'}a_{s'+1}a_{s'+2})^{t'}=0
\end{align*}
\begin{align*}
    &(\psi\star \beta)(e_l\ot e_l)=(\psi\ot \beta)\Delta'(e_l\ot e_l)\\
    &=(\psi\ot \beta)[e_l\ot e_l\ott e_l\ot e_l + \sum_{w\neq e_l} w\ot w\ott e_l\ot e_l + e_l\ot e_l\ott w\ot w]\\
    &=(\psi(f_1\ot f_1) \ot \beta(f_1\ot f_1))=f_1f_1=f_1, \text{so we have} \; (\psi\star \beta)=\psi.
\end{align*}
\begin{align*}
    &(\eta\star \alpha_{s'}^{t'})(e_0\ot e_2)=(\eta\ot \alpha_{s'}^{t'})\Delta'(e_0\ot e_2)\\
    &=(\eta\ot \alpha_{s'}^{t'})[(e_0\ot e_0)\ott (e_0\ot e_{2})+(e_0\ot e_{2})\ott (e_{2}\ot e_{2})]\\
    &+(e_{2}\ot e_{2})\ott (e_0\ot e_{2})+ (e_0\ot e_{2})\ott (e_{0}\ot e_{0})+\\
    &\sum_{w\neq e_0, e_{2}} (w\ot w)\ott (e_l\ot e_l) + (e_l\ot e_l) \ott (w\ot w)]\\
    &=(\eta(e_0\ot e_2)\ot \alpha_{s'}^{t'}(e_2\ot e_2) + \eta(e_0\ot e_2) \ot \alpha_{s'}^{t'}(e_0\ot e_0)+\eta(e_0\ot e_2)\ot \alpha_{s'}^{t'}(e_1\ot e_1)),\\
&\text{so we have  }\;(\eta\star \alpha_{s'}^{t'})= \eta, (i'=2)\; \text{from the following cases}.
\end{align*}
\begin{itemize}
\item Case 1: If $s'=2$ we have 
 $(a_0a_1a_2)^na_0a_1(a_2a_0a_1)^{t'}=(a_0a_1a_2)^{n+t'}a_0a_1=\eta(e_0\ot e_2).$
\item Case 2: If $s'= 1, 3$, we get $0$.
\end{itemize}
\begin{align*}
    &(\eta \star \beta)(f_1\ot f_1)=(\eta \ot \beta)\Delta'(f_1\ot f_1)\\
    &=(\eta \ot \beta)[(f_1\ot f_1)\ott (f_1\ot f_1)\\
    &+\sum_{w\neq f_1}(w\ot w)\ott (f_1\ot f_1)+(f_1\ot f_1)\ott (w\ot w)] &=0
\end{align*}
\begin{align*}
    &(\nu_1\star \alpha_{s'}^{t'})(f_1\ot e_2)=(\nu_1\ot \alpha_{s'}^{t'})\Delta'(f_1\ot e_2)\\
    &=(\nu_1\ot \alpha_{s'}^{t'})[(f_1\ot f_1)\ott (f_1\ot e_2) + (f_1\ot e_2)\ott (f_1\ot f_1) + \\
    &(f_1\ot e_2)\ott (e_2\ot e_2) + (e_2\ot e_2)\ott (f_1\ot e_2)\\
    &+ \sum_{w\neq f_1, e_2} (w\ot w) \ott (e_l\ot e_l) + (e_l\ot e_l) \ott (w\ot w)]\\
    &=((\nu_1(f_1\ot e_2)\ott \alpha_{s'}^{t'}(e_2\ot e_2))=b_1(a_2a_{0}a_{1})^t=b_1, \;\text{ so }(\nu_1\star \alpha_{s'}^{t'}) = \nu_1, (s'=2, t'=0).
\end{align*}
\begin{align*}
    &(\nu_1\star \beta)(f_1\ot e_2)=(\nu_1\ot \beta)\Delta'(f_1\ot e_2)\\
    &=\pi(\nu_1\ot \beta)[(f_1\ot f_1)\ott (f_1\ot e_2) + (f_1\ot e_2)\ott (f_1\ot f_1) + \\
    &(f_1\ot e_2)\ott (e_2\ot e_2) + (e_2\ot e_2)\ott (f_1\ot e_2)+ \sum_{w\neq f_1, e_2} (w\ot w) \ott (e_l\ot e_l) + (e_l\ot e_l) \ott (w\ot w)]\\
    &=(\nu_1(f_1\ot e_2)\ott \beta(f_1\ot f_1))=b_1f_1=0, \text{ so }\;(\nu_1\star \beta)=0.
\end{align*}
\end{proof}


\end{document}